\documentclass[11pt]{article}
\usepackage{yub}

\newcommand{\symm}{\mc{S}}

\newcommand{\rgrad}{{\mathop{\rm grad\,}}}
\newcommand{\rhess}{{\mathop{\rm Hess\,}}}
\newcommand{\proj}{{\mc{P}}}
\newcommand{\gm}{{\mathsf{G}}}

\def\smoothf{\beta_f}
\def\smoothF{\beta_F}
\def\hesssmoothf{\rho_f}
\def\smoothphif{\beta_{\vphi_f}}
\def\hesssmoothphif{\rho_{\vphi_f}}
\def\lipf{L_f}
\def\lipphif{L_{\vphi_f}}
\def\smoothc{\beta_c}
\def\sigmacq{\gamma}
\def\sigmadn{\gamma_Q}
\def\eigvaldual{\lambda_{Z_\star}}
\def\Lambdaall{\Lambda}
\def\Lambdaqg{\Lambda_{\rm qg}}
\def\Lambdaproj{\Lambda_{\rm proj}}
\def\sigmamax{\sigma_{\max}(M_{R^\star})}

\def\prox{{\bf prox}}
\def\dl{{\Vert}}

\begin{document}
\title{Proximal algorithms for constrained composite optimization, \\
  with applications to solving low-rank SDPs}
\author{Yu Bai\footnote{Department of Statistics, Stanford
    University. {\tt yub@stanford.edu}}
  \and John Duchi\footnote{Department of Statistics and Electrical
    Engineering, Stanford University. {\tt jduchi@stanford.edu}}
  \and Song Mei\footnote{Institute for Computational and Mathematical
    Engineering, Stanford University. {\tt songmei@stanford.edu}}
}

\maketitle

\begin{abstract}
  We study a family of (potentially non-convex)
  constrained optimization problems
  with convex composite structure.
  Through a novel analysis of non-smooth geometry, we show that
  proximal-type algorithms applied to exact penalty formulations
  of such problems exhibit local linear convergence under
  a quadratic growth condition, which
  the compositional structure we consider ensures.


  The main application of our results is to low-rank semidefinite
  optimization with Burer-Monteiro factorizations. We precisely identify the
  conditions for quadratic growth in the factorized problem via structures
  in the semidefinite problem, which could be of independent interest for
  understanding matrix factorization.


  

\end{abstract}

\section{Introduction}
We consider constrained composite optimization problems of the form
\begin{equation}
  \label{problem:constrained-composite}
  \begin{aligned}
    \minimize & ~~ f(c(x))\\
    \subjectto & ~~ Ac(x) - b = 0,
  \end{aligned}
\end{equation}
in which $c:\R^n\to\R^m$ is a smooth map, $f:\R^m\to\R$ is convex,
$A\in\R^{k\times m}$, and $b\in\R^k$. Our main motivating example for
such problems is the Burer-Monteiro factorization method for
solving the semidefinite optimization problem
\begin{equation}
  \label{problem:constrained-sdp}
  \begin{aligned}
    \minimize & ~~ f(X) \\
    \subjectto & ~~ \mc{A}(X)=b, ~ X \succeq 0,
  \end{aligned}
\end{equation}
where $f:\symm^n\to\R$ is smooth convex, $\mc{A}:\symm^n\to\R^k$ is a
symmetric linear operator with $[\mc{A}(X)]_i = \<A_i,X\>$, and
$b\in\R^k$.  The celebrated \citet{BurerMo03} approach
proposes to solve~\eqref{problem:constrained-sdp} by factorizing
$X=RR^\top$ for $R\in\R^{n\times r}$ and solving
\begin{equation}
  \label{problem:constrained-bm}
  \begin{aligned}
    \minimize & ~~ f(RR^\top) \\
    \subjectto & ~~ \mc{A}(RR^\top) = b.
  \end{aligned}
\end{equation}
This is an instance of the
problem~\eqref{problem:constrained-composite} with $c(R)=RR^\top$.
When the solution $X_\star$ of~\eqref{problem:constrained-sdp} has low
rank, satisfying
\begin{equation*}
  X_\star = R_\star R_\star^\top
\end{equation*}
for some $R_\star\in\R^{n\times r_\star}$ with $r_\star\ll n$, the
Burer-Monteiro factorization is particularly appealing because, in addition
to its lower storage and computational cost, it can solve the original
problem~\eqref{problem:constrained-sdp}.
Many problems in science and engineering can be cast as
problem~\eqref{problem:constrained-sdp} with a low-rank solution,
including phase retrieval~\cite{CandesStVo13,WaldspurgerDaMa15},
community detection~\cite{AbbeBaHa16,HajekWuXu16},
phase synchronization~\cite{BandeiraBoSi17,Singer11},
and robust PCA~\cite{McCoyTr11}.


Problem~\eqref{problem:constrained-composite} is the constrained
variant of \emph{composite optimization}
problems~\cite{Burke85},
a family of structured non-convex (and potentially non-smooth)
optimization problems of the form
\begin{equation}
  \label{problem:composite}
  \minimize ~~ \vphi(x) = h(c(x)),
\end{equation}
where $c:\R^n\to\R^m$ is smooth and $h:\R^m\to\R$ is convex. Such
composite structure appears naturally in learning problems with
a convex but non-smooth loss, such as robust phase
retrieval~\cite{DuchiRu18,DavisDrPa19}. 

A prevailing algorithm for solving the composite
problem~\eqref{problem:composite} is the \emph{prox-linear
  algorithm}~\cite{Burke85,BurkeFe95}, which sequentially minimizes a
Taylor-like model of $\vphi$:
\begin{equation}
  \label{algorithm:prox-linear}
  \begin{aligned}
    &  x^{k+1} = \argmin_{x\in\R^n} \set{\vphi(x^k; x) +
      \frac{1}{2t}\ltwo{x - x^k}^2}, \\
    & {\rm where}~~\vphi(x^k; x) \defeq  h(c(x^k) + \grad c(x^k)^\top(x-x^k)).
  \end{aligned}
\end{equation}
The model function $x\mapsto \vphi(x^k;x)$ linearizes the smooth map
$c$, keeps the outer convex function $h$, and is therefore
convex. Each iteration thus requires solving a $1/t$-strongly convex
problem, which is (frequently) efficient.

When $h$ is Lipschitz and $c$ is smooth, the prox-linear algorithm has
global convergence to a stationary point, measured by the sub-differential
stationarity $\dist(0; \partial\vphi(x))$.  Analysis of its local
convergence has been more sophisticated---local linear convergence has been
established under \emph{tilt stability}, which requires a unique minimizer
and strong growth around it~\cite{DrusvyatskiyLe16}. A
naive transformation of the
constrained problem~\eqref{problem:constrained-composite}
into~\eqref{problem:composite}, taking $h(y) =f(y) +\mathbb{I}[Ay
  -b = 0]$, violates such local analyses as there may be
multiple $x$ with $c(x)=y_\star$. Our seek
alternative approaches for solving~\eqref{problem:constrained-composite} with
local convergence guarantees.

\subsection{Outline and our contribution}
In this paper, we consider the exact penalty method~\citep[cf.][]{BorweinLe06}
for solving
problem~\eqref{problem:constrained-composite},
which translates the constraint into an exact penalty term
$\ltwo{Ac(x)-b}$ and solve the unconstrained problem
\begin{equation}
  \label{problem:penalized-composite}
  \minimize ~~ \vphi(x) \defeq f(c(x)) + \lambda\ltwo{Ac(x)-b}.
\end{equation}
In the matrix problem, this corresponds to solving
\begin{equation}
  \label{problem:penalized-bm}
  \minimize ~~ \vphi(R) \defeq f(RR^\top) +
  \lambda\ltwo{\mc{A}(RR^\top)-b}.
\end{equation}
Such an exact penalty term encourages $x$ to fall onto the
constraint set, and as the norm $\ltwo{\cdot}$ grows linearly in
$Ac(x)-b$, one expects that it has a dominating effect over the
objective $f(c(x))$ when $x$ is not on the set, therefore penalizing
infeasible $x$'s~\cite{Bertsekas99}.

Problem~\eqref{problem:penalized-composite} is a composite optimization
problem, and we therefore study the convergence of the prox-linear
algorithm~\eqref{algorithm:prox-linear}, providing arguments on its
convergence for matrix problems of the form~\eqref{problem:penalized-bm}.

We summarize our contributions.
\begin{itemize}
\item We define \emph{norm convexity}, a local geometric property for
  generic non-smooth functions. Norm convexity is a weak notion of
  local convexity, and it dovetails with other regularity conditions for
  composite optimization (e.g.\ sub-differential reguality).  We show
  that the exact penalty function~\eqref{problem:penalized-composite}
  satisfies norm convexity if it has \emph{quadratic growth} around
  the (local) minimizing set (Section~\ref{section:geometry}).
\item We show that the prox-linear algorithm has local linear
  convergence for generic composite optimization if $\vphi$ has
  quadratic growth and norm convexity. Our result
  extends~\citet{DrusvyatskiyLe16} and does not rely on the tilt
  stability assumption required there. Consequently, the prox-linear
  algorithm on the exact penalty
  function~\eqref{problem:penalized-composite} has local linear
  convergence as long as the problem has quadratic growth
  (Section~\ref{section:algorithm}).
\item We instantiate our result on the factorized matrix
  problem~\eqref{problem:penalized-bm}. To verify the assumptions for
  the convergence result, we study whether the quadratic growth
  of~\eqref{problem:penalized-composite} can be deduced from the
  quadratic growth of the original
  SDP~\eqref{problem:constrained-sdp}. We show that quadratic
  growth is always preserved if the rank is exact
  ($r=r_\star$), and is preserved for
  linear objectives ($f(X) = \<C, X\>$) when
  the rank is over-specified ($r>r_\star$). In contrast, when $f$
  is non-linear, quadratic growth is no longer preserved under rank
  over-specification (Section~\ref{section:matrix-growth}). This gives
  a precise characterization for the convergence of the prox-linear
  algorithm on factorized SDPs and could be of broader interest
  for understanding matrix factorization.
\item We provide concrete examples of matrix problems on which our
  theory is applicable (Appendix~\ref{section:examples}) and numerical
  experiments verifying our convergence results
  (Appendix~\ref{section:experiments}).
\end{itemize}
We provide a roadmap of our main results in
Figure~\ref{figure:flowchart}.

\begin{figure*}
  \centering
  \includegraphics[width=0.9\linewidth]{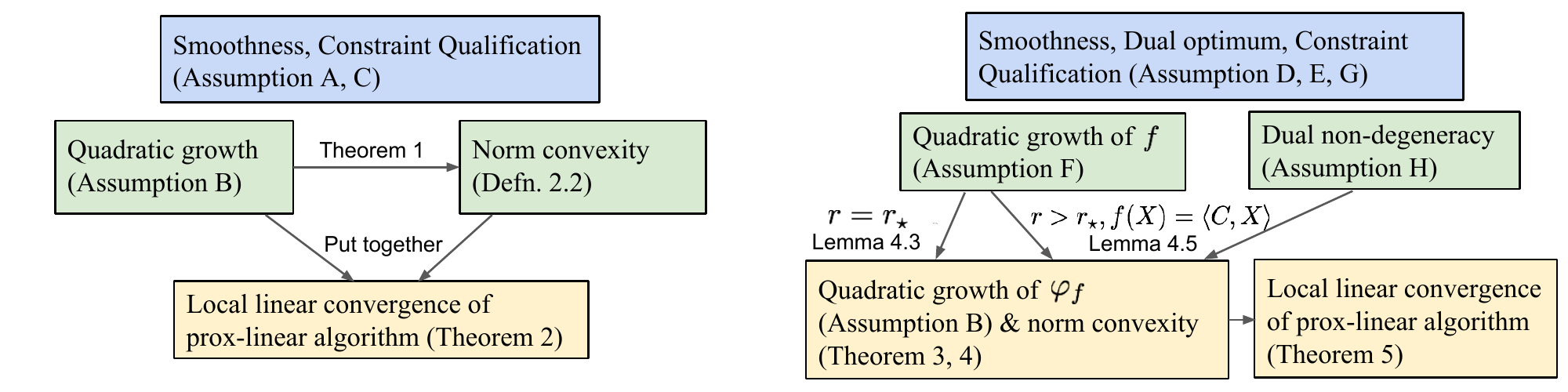}
  \caption{\small A roadmap of the main results. Left: results for generic
    constrained composite optimization (Section~\ref{section:geometry}
  and~\ref{section:algorithm}). Right: application on factorized matrix
  problems (Section~\ref{section:matrix-growth}).} 
  \label{figure:flowchart}
\end{figure*}

\subsection{Related work}
\paragraph{Composite optimization}
The exact penalty method was one particular early motivation for considering
convex composite functions~\cite{Burke91}. The
workof~\citet{Burke85,LewisWr08,DrusvyatskiyLe16,DrusvyatskiyIoLe16} studies
the convergence of proximal algorithms on composite problems. In particular,
\citet{DrusvyatskiyLe16} establish local linear convergence of the
prox-linear algorithm, which shows the ``natural'' rate of convergence in
presence of sub-differential regularity and the tilt-stability condition,
and poses the question (in its Section 5) whether sub-differential
regularity is implied by quadratic growth for general convex composite
problems. We estbalish the norm convexity condition under quadratic growth,
thereby resolving the problem in the special case of penalized objectives.

\paragraph{Matrix factorization}
The idea of solving SDPs by factorizing $X=RR^\top$ is due to
\citet{BurerMo03}. This factorization encodes the PSD constraint $X\succeq
0$ into the objective, at the cost of turning the problem non-convex. When
$f(X)=\<C,X\>$, the enlightening result of~\citet{Pataki98} shows that any
SDP with $k$ linear constraints always has a solution satisfying
$\rank(X_\star)\le \sqrt{2k}$, thus when the number of constraints is small,
one can always set $R\in\R^{n\times r}$ wih $r\ge\sqrt{2k}$, thereby saving
huge computational and storage cost. The non-convexity of the factorized
problem~\eqref{problem:constrained-bm} makes possible spurious local
minimizers and weird geometries. \citet{BoumalVoBa16} establish benign
geometry in a general case, showing that for generic $C$, all second-order
critical points of the linear problem~\eqref{problem:constrained-bm} are
global minima as long as the rank is overspecified (consistent with our
results). For some special problems with non-linear $f$ such as matrix
completion and low-rank matrix sensing, a recent line of
work~\cite{GeLeMa16,GeJiZh17} shows often there is no spurious local
minimum. 

\paragraph{Alternative methods for semidefinite optimization}
The majority of early development on SDPs focuses on
interior-point algorithms, established by~\citet{NesterovNe94}
and~\citet{Alizadeh95}. Interior-point methods are
efficient and robust on small-scale problems ($n\lesssim 10^3$)
but quickly become infeasible beyond that, as they must compute
matrix inverses or Schur complements. Augmented Lagrangian methods,
including ADMM and Newton-CG algorithms,
appear to be faster and more scalable, with well-developed
software available (e.g.\ {\tt SDPAD}~\cite{WenGoYi10} and {\tt
  SDPNAL}~\cite{ZhaoSuTo10,YangSuTo15}).

Riemannian methods are suitable for solving
problem~\eqref{problem:constrained-bm} when the constraint
$\mc{A}(RR^\top)=b$ has special structures such as block-diagonal
constraints or orthogonality constraints~\cite{Boumal15}. These
methods are very efficient in practice. Results on their
local~\cite{AbsilMaSe09} and global convergence~\cite{BoumalAbCa16}
are present. For a thorough introduction to Riemannian optimization on
matrix manifolds see the book of~\citet{AbsilMaSe09}.

\subsection{Notation}
We usually reserve letters $x,y,a,b,\dots$ for vector variables and 
capital letters $X,Y,A,B,\dots$ for matrices. The space of $n\times n$
symmetric matrices is $\symm^n$. For a matrix $A$, 
we let $\lambda_i(A)$ and $\sigma_i(A)$ denote the eigenvalues /
singular values of $A$ sorted in decreasing order. The two norm,
Frobenius norm, and operator norm  are denoted by
$\ltwo{\cdot}$, $\lfro{\cdot}$, and $\opnorm{\cdot}$. For twice
differentiable $f:\R^n\to\R$, $\grad f$ and $\grad^2 f$ denote its
gradient and Hessian. For vector-valued function $c:\R^n\to\R^m$, let
$\grad c(x)\in\R^{n\times m}$ be the (tranposed) Jacobian, so that the
first-order Taylor expansion reads
\begin{equation*}
  c(x) = c(y) + \grad c(y)^\top(y-x) + o(\norm{y-x}).
\end{equation*}
For $\vphi(\cdot)=h(c(\cdot))$ convex composite,
where $c:\R^n\to\R^m$ is smooth and $h:\R^m\to\R$ is convex, let
$\partial\vphi$ denote its (Frechet) sub-differential
\begin{equation*}
  \partial\vphi(x) = \set{\grad c(x) g:g\in\partial h(c(x))}
  \subset \R^n.
\end{equation*}
We let
$\dist(x,S) = \inf_{y\in S}\ltwo{x-y}$ be the distance
of $x$ to $S$ and $\mc{N}(S,\eps) = \set{x:\dist(x,S)\le \eps}$
denote the $\eps$-neighborhood of $S$.

\section{Geometry of the composite objective}
\label{section:geometry}
In this section, we analyze the local geometry of $\vphi$, the
composite objective. We first show that quadratic growth
(Definition~\ref{def:quadratic-growth}) is preserved
when we reformulate problem~\eqref{problem:constrained-composite}
to~\eqref{problem:penalized-composite}. We then show such quadratic
growth will imply norm convexity and sub-differential regularity
(Definition~\ref{def:norm-convexity} and~\ref{def:subregularity}).

Throughout the paper, we let
\begin{equation*}
  \vphi_f(x) = f(c(x))~~~{\rm and}~~~\vphi_g(x)=\ltwo{Ac(x)-b}
\end{equation*}
denote the objective and the penalty function, and
\begin{equation}
  \label{equation:definition-vphi}
  \vphi(x)=\vphi_\lambda(x) \defeq \vphi_f(x)+\lambda\vphi_g(x)
\end{equation}
denote the penalized objective (the subscript $\lambda$ is omitted
when it is clear from the context). Let $S$ be a local minimizing set
of Problem~\eqref{problem:constrained-composite},
i.e. $\vphi_f(x_\star)=\vphi_\star$ for any $x_\star\in S$ and
$\vphi_f(x)\ge\vphi_\star$ for all $x\in\mc{N}(S,\eps_0)$ such that
$Ac(x)-b=0$.

A first question will be whether minimizing $\vphi$ is equivalent to
solving the original constrained
problem~\eqref{problem:constrained-composite}, i.e.~whether $S$ is
also the local minimizing set of $\vphi$. On the constraint set,
$\vphi_f=\vphi$, so the minimizing set is $S$; off the constraint set,
the term $\lambda\ltwo{Ac(x)-b}$ has a ``pointy'' behavior and will
produce strong growth, so if $\vphi_f$ is sufficiently smooth,
intuitively this penalty term will dominate and force $\vphi$ to also
grow off the constraint set. We will make this argument precise in
Section~\ref{section:preservation-quadratic-growth} and give an
affirmative answer under quadratic growth and constraint
qualification. 

We now define the quadratic growth property.
\begin{definition}[Quadratic growth]
  \label{def:quadratic-growth}
  A function $f:\R^n\to\R$ is said to have $\alpha$-quadratic growth
  in $\mc{X}$ around a local minimizing set $S$ if
  \begin{equation*}
    f(x) \ge f_\star + \frac{\alpha}{2}\dist^2(x, S),~~\forall
    x\in\mc{X},
  \end{equation*}
  where $f_\star = f(S)$ is the function value on $S$.
\end{definition}

We now collect our assumptions for this section. For properties that
are required to hold locally, we assume there exists an $\eps_0>0$
such that all local properties hold in $\mc{N}(S,\eps_0)$.

\begin{assumption}[Smoothness]
  \label{assumption:smoothness}
  In a neighborhood $\mc{N}(S,\eps_0)$, 
  the objective $\vphi_f\in C^2$ with $\smoothphif$-bounded and
  $\hesssmoothphif$-Lipschitz Hessian. That is, 
  \begin{equation*}
    \begin{aligned}
      & \opnorm{\grad^2 \vphi_f(x)} \le \smoothphif ~~~{\rm
        and} \\
      & \opnorm{\grad^2 
        \vphi_f(x_1) - \grad^2 \vphi_f(x_2)} \le
      \hesssmoothphif\ltwo{x_1-x_2}.
    \end{aligned}
  \end{equation*}
  Further, $\vphi_f$ is $\lipphif$-locally Lipschitz.
  Functions $f$ and $c$ are also smooth with
  parameters accordingly (for example, $c$ is $\smoothc$-smooth).
\end{assumption}

\begin{assumption}[Quadratic growth]
  \label{assumption:quadratic-growth}
  There exists some $\alpha_{\vphi_f}$ such that locally
  \begin{equation*}
    \vphi_f(x) \ge \vphi_f(x_\star) + \frac{\alpha_{\vphi_f}}{2}
    \dist(x, S)^2
  \end{equation*}
  holds for all $x\in\mc{N}(S,\eps_0)$ such that $Ac(x)-b=0$.
\end{assumption}
Assumption~\ref{assumption:quadratic-growth} ensures that the
constrained optimization problem has quadratic growth around the
minimizing set $S$. This will be the main assumption that we hinge on
to show various geometric properties.

An additional assumption that we make on the constrained
problem~\eqref{problem:constrained-composite} is the following. 
\begin{assumption}[Constraint qualification]
  \label{assumption:constraint-non-degeneracy}
  There exists some constant $\sigmacq>0$ such that for all
  $x\in\mc{N}(S,\eps_0)$, the Jacobian of the constraint function
  $A\grad c(x)^\top$ has full row rank and satisfy the quantitative
  bound $\sigma_{\min}(A\grad c(x)^\top)\ge \sigmacq$. Consequently,
  in the neighborhood $\mc{N}(S,\eps_0)$, the constraint set
  \begin{equation*}
    \mc{M} \defeq \set{x\in\R^n:Ac(x)=b}
  \end{equation*}
  is a smooth manifold. We further assume that the minimizing set $S$
  is a compact smooth submanifold of $\mc{M}$.
\end{assumption}
Assumption~\ref{assumption:constraint-non-degeneracy}
is known as the Linear Independence
Constraint Qualification (LICQ)
in the nonlinear programming literature and requires that the normal
space of $\mc{M}$ at $x$ is $m$-dimensional. (The
reader can refer to~\citep[Chapter 3]{AbsilMaSe09} for background on
smooth manifolds.)


\subsection{Preservation of quadratic growth}
\label{section:preservation-quadratic-growth}
We start by asking the following question: does the penalty method
preserve quadratic growth? Namely, if a constrained problem has
quadratic growth on the constraint set, does the penalized objective
have quadratic growth in the whole space? The following result gives
an affirmative answer. 
\begin{lemma}
  \label{lemma:penalized-quadratic-growth}
  Let Assumptions~\ref{assumption:smoothness}
  and~\ref{assumption:constraint-non-degeneracy}
  hold. Let Assumption~\ref{assumption:quadratic-growth} hold, i.e.
  for all $x\in\mc{N}(S,\eps_0)$ such that $Ac(x)=b$,
  \begin{equation*}
    \vphi_f(x) \ge \vphi_f(x_\star) + \frac{\alpha_{\vphi_f}}{2}
    \dist(x,S)^2,
  \end{equation*}
  then the penalized objective $\vphi=\vphi_f+\lambda\vphi_g$ has
  local quadratic growth: for any $\delta\in(0,1)$, setting
  $\Lambdaqg=\lipphif/(\delta\sigmacq)$, there exists a neighborhood
  $\mc{N}(S,\eps)$ such that for all $\lambda\ge\Lambdaqg$ and
  $x\in\mc{N}(S,\eps)$, we have
  \begin{equation*}
    \vphi(x) \ge \vphi(x_\star) + \frac{\alpha_\vphi}{2}
    \dist(x,S)^2,
  \end{equation*}
  where $\alpha_\vphi=(1-\delta)\alpha_{\vphi_f}$.
\end{lemma}
Lemma~\ref{lemma:penalized-quadratic-growth} says that quadratic
growth is preserved in the penalized formulation. In particular,
$x_\star$ is a local minimum of $\vphi$. The proof can be found in
Appendix~\ref{appendix:proof-penalize-growth}. 

\subsection{Quadratic growth implies norm convexity}
We now define norm convexity and sub-differential regularity, two
geometric properties that are essential to establishing the convergence of proximal
algorithms in Section~\ref{section:algorithm}.
\begin{definition}[Norm convexity]
  \label{def:norm-convexity}
  The function $\vphi$ is \emph{norm convex} around the minimizing 
  set $S$ with constant $\ell>0$ if for all $x$ near $S$, we have
  \begin{equation*}
    \vphi(x) - \vphi_\star \le \ell \cdot \dist(0,\partial\vphi(x))
    \cdot \dist(x,S).
  \end{equation*}
\end{definition}

Next, we recall the definition of sub-differential
regularity~\cite{DrusvyatskiyLe16}.
\begin{definition}[Sub-differential regularity]
  \label{def:subregularity}
  The function $\vphi$ is $\ell$-subdifferentialy regular at $S$ if
  for all $x$ near $S$, we have
  \begin{equation*}
    \dist(x, S) \le \ell \cdot \dist(0, \partial\vphi(x)).
  \end{equation*}
\end{definition}

We now present our main geometric result, that is, for $\vphi$ having
the penalty structure~\eqref{equation:definition-vphi}, quadratic growth
of $\vphi$ implies norm convexity and sub-differential regularity.

To gain some intuition, let us illustrate that quadratic growth
implies norm convexity and sub-differential regularity in the convex case.
For $\vphi(x)$ convex, assuming $\alpha$-quadratic growth, we have
\begin{equation*}
  \begin{aligned}
    & \quad \frac{\alpha}{2}\dist(x, S)^2 \le \vphi(x) - \vphi(x_\star) \le
    \<\partial\vphi(x), x-x_\star\> \\
    & \le \norm{\partial\vphi(x)} \norm{x-x_\star}
  \end{aligned}
\end{equation*}
for any choice of the subgradient $\partial\vphi(x)$ and minimum
$x_\star$. Choosing the minimum norm subgradient and $x_\star$ closest
to $x$ gives sub-differential regularity
\begin{equation*}
  \dist(x, S) \le \frac{2}{\alpha}\dist(0, \partial\vphi(x)).
\end{equation*}
Therefore norm convexity generalizes convexity by only requiring
$\vphi(x) - \vphi(x_\star)\le C\cdot \norm{\partial \vphi(x)}\cdot
\ltwo{x-x_\star}$.

Norm convexity specifies a local regularity condition for non-convex
non-smooth functions. While it does not necessarily hold for general
composite functions, we show that for our exact penalty objective
$\vphi$, quadratic growth does imply norm convexity and hence
sub-differential regularity.


\begin{theorem}[Norm convexity and sub-differential regularity]
  \label{theorem:quadratic-implies-regularity}
  Let Assumptions~\ref{assumption:smoothness},
  ~\ref{assumption:quadratic-growth}
  and~\ref{assumption:constraint-non-degeneracy} hold.
  Then, there exist a constant $\ell\le 5$ and
  a neighborhood $\mc{N}(S,\eps)$ in which setting the penalty
  parameter $\lambda \ge \Lambdaall$, we have norm convexity 
  \begin{equation}
    \label{equation:comparison}
    \vphi(x) - \vphi_\star \le \ell \cdot \dist(0, \partial\vphi(x))
    \cdot \dist(x, S).
  \end{equation}
  Consequently, sub-differential regularity holds:
  \begin{equation}
    \label{equation:regularity}
    \dist(x, S) \le  \wt{\ell} \cdot \dist(0, \partial\vphi(x)),
  \end{equation}
  where $\wt{\ell} = 4\ell/\alpha_{\vphi_f}$.
\end{theorem}
The proof can be found in Appendix~\ref{appendix:proof-main}.



\section{Local convergence of the prox-linear algorithm}
\label{section:algorithm}
We now analyze the convergence of the prox-linear
algorithm~\eqref{algorithm:prox-linear} for generic composite problems
of the form~\eqref{problem:composite}.




Recall that the prox-linear algorithm iterates
\begin{equation*}
    x^{k+1} = \argmin_{x\in\R^n}\vphi_t(x^k;x),
\end{equation*}
where
\begin{equation*}
  \begin{aligned}
    & \vphi_t(x^k; x) = \vphi(x^k;x) + \frac{1}{2t}\ltwo{x-x^k}^2, \\
    & \vphi(x^k; x) = h(c(x^k) + \grad c(x^k)^\top(x-x^k)),
  \end{aligned}
\end{equation*}
and $t>0$ is a small stepsize.

For the convergence result, we assume that $h$ is $L$-Lipschitz and
$c$ is $\beta$-smooth, i.e.
\begin{equation*}
  \ltwo{\grad c(x_1) - \grad c(x_2)} \le
  \beta\ltwo{x_1-x_2},~\textrm{for all}~x_1,x_2.
\end{equation*}
An immediate consequence of the smoothness is that
\begin{equation}
  \label{equation:quadratic-approximation}
  \left| \vphi(x_0;x) - \vphi(x) \right| \le
  \frac{L\beta}{2}\ltwo{x-x_0}^2,
\end{equation}
so $\vphi(x_0;x)$ gives a local quadratic approximation of $\vphi(x)$
near $x_0$. In particular, when $t\le(L\beta)^{-1}$,
$\vphi_t(x_0;x)$ is an upper bound on $\vphi(x)$, implying that the
prox-linear algorithm is a descent method.


We now present our main algorithmic result: the
prox-linear algorithm has local linear convergence as long as $\vphi$
has quadratic growth and norm convexity.
\begin{theorem}[Local linear convergence of the prox-linear algorithm]
  \label{theorem:linear-convergence}
  Suppose $\vphi$ satisfies the above assumptions and has a compact
  local minimizing set $S$. Assume that there exists
  $\eps_0>0$ such that the following happens in $\mc{N}(S,\eps_0)$:
  \vspace{-0.3cm}
  \begin{enumerate}[(1)]
  \item $\vphi$ has $\alpha_\vphi$-quadratic growth
    and norm convexity with constant $\ell$ around $S$;
    \vspace{-0.2cm}
  \item Prox-linear iterates has the proximity property (see
    Definition~\ref{definition:proximity} for a formal definition and
    discussion).
  \end{enumerate}
  \vspace{-0.3cm}
  Then, for $x^0$ sufficiently close to $S$,
  the prox-linear algorithm~\eqref{algorithm:prox-linear} has linear
  convergence:
  \begin{equation*}
    \vphi(x^{k+1}) - \vphi_\star \le q(\vphi(x^k) - \vphi_\star),
  \end{equation*}
  where
  \begin{equation*}
    q \defeq 1 - \frac{1}{9+40\ell+100\ell^2L\beta/\alpha_\vphi}.
  \end{equation*}
\end{theorem}
The proof builds on existing results on composite optimization
from~\citep[Section 5,6]{DrusvyatskiyLe16}, which we review in
Appendix~\ref{appendix:review-composite-opt}, and makes novel use of
norm convexity to establish the local linear convergence.
The proof can be found in
Appendix~\ref{appendix:proof-prox-linear}. Relationship between our
result and existing results based on tilt stability is
discussed in Appendix~\ref{appendix:relationship-tilt-stability}.

To apply Theorem~\ref{theorem:linear-convergence} on the exact penalty
formulation~\eqref{problem:penalized-composite}, we only need to verify
the quadratic growth assumption (as norm convexity is then implied by
Theorem~\ref{theorem:quadratic-implies-regularity}) -- we will see
more on this on matrix problems in
Section~\ref{section:matrix-growth}.

\section{Application on factorized matrix problems}
\label{section:matrix-growth}
We now instantiate our geometry and convergence results
on our main application, the matrix
problem~\eqref{problem:penalized-bm}, and show that the prox-linear
method achieves local linear convergence for solving many low-rank
semidefinite problems.

For the matrix problem, recall that $c(R)=RR^\top$ and
\begin{equation*}
  \vphi_f(R) =
  f(RR^\top),~~\vphi_g(R)=\ltwo{\mc{A}(RR^\top)-b}.
\end{equation*}


Recall that our algorithmic result
(Theorem~\ref{theorem:linear-convergence}) requires $\vphi$ to have
quadratic growth. The main focus of this section is to study whether
the quadratic growth of $\vphi_f(R)$
(Assumption~\ref{assumption:quadratic-growth}) can be deduced from the
quadratic growth of $f(X)$, which can often be verified more
straightforwardly.


We build such connection in two separate cases: factorization with the
exact rank ($r=r_\star$) and with rank over-specification
($r>r_\star$). We show that quadratic growth is transferred from $f$
to $\vphi_f$ only if we have the exact rank
(Section~\ref{section:matrix-growth-exact-rank}), or that we
over-specify the rank and $f$ is linear ($f(X)=\<C,X\>$)
(Section~\ref{section:over-specification-succeeds}).  In both cases, we adapt
Theorem~\ref{theorem:linear-convergence} to provide the convergence
result as a corollary. If $f$ is not linear, quadratic growth will in
general fail to hold on $\vphi_f$ and the prox-linear algorithm no
longer have local linear convergence. We demonstrate this via a
counter-example (Section~\ref{section:over-specification-fails}).

Throughout this section, we assume the following assumptions on the
semidefinite objective $f$, which we will use to deduce properties
on the factorized objective $\vphi_f$. 
\begin{assumption}[Smoothness]
  \label{assumption:smooth-f}
  The objective $f\in C^2$ with $\smoothf$-bounded and
  $\hesssmoothf$-Lipschitz Hessian, i.e. for all
  $X\succeq 0$,
  \begin{equation*}
    \begin{aligned}
      & \opnorm{\grad^2 f(X)} \le \smoothf~~~{\rm
        and}\\
      & \opnorm{\grad^2 f(X_1) - \grad^2 f(X_2)} \le
      \hesssmoothf\lfro{X_1-X_2}.
    \end{aligned}
  \end{equation*}
  Further, $f$ is $\lipf$-locally Lipschitz near $X_\star$.
\end{assumption}

\begin{assumption}[Dual optimum]
  \label{assumption:dual-existence}
  There exists at least one dual optimum $(y_\star, Z_\star)$
  associated with $X_\star$, i.e. $y_\star\in\R^k$ and $Z_\star\succeq
  0$ that satisfy the KKT conditions
  \begin{equation}
    \label{equation:dual-optimum}
    \grad f(X_\star) + \sum_{i=1}^k y_iA_i - Z_\star = 0~~~{\rm
      and}~~~\<Z_\star, X_\star\> = 0.
  \end{equation}
\end{assumption}
\begin{assumption}[Rank-$r$ quadratic growth]
  \label{assumption:low-rank-quadratic-growth}
  There exists $\eps>0$ such that 
  \begin{equation*}
    f(X) \ge f(X_\star) + \frac{\alpha_f}{2} \cdot \lfro{X-X_\star}^2
  \end{equation*}
  for all $X$ in the set
  \begin{equation*}
    \set{X:X\succeq 0,~\mc{A}(X)=b,~\rank(X)\le r} \cap
    \ball(X_\star,\eps).
  \end{equation*}
\end{assumption}
Assumption~\ref{assumption:dual-existence} states that the
semidefinite problem has a unique low-rank solution and that there is
no duality gap. A number of conditions such as the Slater's condition
guarantee no duality gap and such dual optimum must exist. 
Assumption~\ref{assumption:low-rank-quadratic-growth} ensures that on
low-rank feasible points, $f$ has strong growth around $X_\star$.  As
we only require strong growth on low-rank matrices, this assumption is
often more likely to hold than full quadratic growth. We will
demonstrate this in Section~\ref{section:examples}.
\begin{assumption}[Rank $r$-constraint qualification]
  \label{assumption:constraint-qualification-matrix}
  The constraint set
  \begin{equation*}
    \mc{M}_r \defeq \set{R\in\R^{n\times r}: \mc{A}(RR^\top)=b}
  \end{equation*}
  is a smooth manifold.
  There exists some constant $\sigmacq>0$ such that the constraint
  coefficient matrices $A_1,\dots,A_k$ satisfy 
  $\sigma_{\min}(M_R)\ge\sigmacq$ for all $R$ in a neighborhood of $S$,
  where $M_R\in\R^{nr\times k}$ is defined via
  \begin{equation*}
    M_R \defeq \begin{bmatrix}
      \vec(A_1R) ~ \dots ~ \vec(A_kR)
    \end{bmatrix}.
\end{equation*}
\end{assumption}

\subsection{Preliminaries: global optimality, matrix distance}

Let $X_\star=R_\star R_\star^\top$ be the unique minimum of
problem~\eqref{problem:constrained-sdp} and
$R_\star\in\R^{n\times r_\star}$.

\paragraph{Global optimality}
We first show that the (global) minimizing set of the exact penalty
function corresponds exactly to the solution of the original
semidefinite problem~\eqref{problem:constrained-sdp} when the penalty
parameter $\lambda$ is sufficiently large.
\begin{lemma}
  \label{lemma:global-optimality}
  Under Assumption~\ref{assumption:dual-existence}, when
  $r\ge r_\star$ and $\lambda\ge\ltwo{y_\star}$, the minimizing
  set of $\vphi(R)=f(RR^\top)+\lambda\ltwo{A(RR^\top)-b}$ is
  \begin{equation*}
    S = \set{R\in\R^{n\times r}: RR^\top = X_\star}.
  \end{equation*}
\end{lemma}
The proof can be
found in Appendix~\ref{appendix:proof-global-optimality}.


\paragraph{Non-uniqueness and matrix distance}
The factorization $X=RR^\top$ brings in a great deal of non-uniqueness
issues, but this non-uniqueness acts fairly nicely and satisfies the
assumptions we require in our geometry and optimization results. One
particular nice property is that the minimizing set $S$, having the form
\begin{equation*}
  S = \set{R_\star\Omega: \Omega\in\R^{r\times r}, \Omega^\top\Omega =
    I_r},
\end{equation*}
is a compact smooth manifold, and the distance
\begin{equation*}
  \dist(R, S) = \min_{\Omega\in\R^{r\times r},\Omega^\top\Omega = I_r}
  \lfro{R - R_\star\Omega}, 
\end{equation*}
is the Procrustes distance between $R$ and $R_\star$. More
background on the factorization map $X=RR^\top$ and the Procrustes
distance are provided in Appendix~\ref{appendix:background-matrix}.


\subsection{Matrix growth: the exact-rank case}
\label{section:matrix-growth-exact-rank}
If we know the exact rank $r_\star$, we could set $r=r_\star$ and
factorize with $R\in\R^{n\times r_\star}$. In this case, the Euclidean
distance in the $X$ space and the $R$ space is nicely connected, as
stated in the following bound.
\begin{lemma}[Lemma 41,~\cite{GeJiZh17}]
  \label{lemma:matrix-factorization-bound}
  Let $R,R_\star\in\R^{n\times r}$ be two matrices such that
  $R^\top R_\star\succeq 0$ (so that they are aligned in the
  Procrustes distance), then
  \begin{equation*}
    \begin{aligned}
      & \quad \lfro{RR^\top - R_\star R_\star^\top}^2 \ge
      2(\sqrt{2}-1)\lfro{R_\star(R-R_\star)^\top}^2 \\
      & \ge 
      2(\sqrt{2}-1)\sigma_r^2(R_\star) \cdot \lfro{R - R_\star}^2.
    \end{aligned}
  \end{equation*}
\end{lemma}
Building on this result, we show that
when $f(X)$ has quadratic growth around $X_\star$ in the
constraint set, so does $\vphi_f(R)=f(RR^\top)$. The proof is in
Appendix~\ref{appendix:proof-factorize-growth}. 
\begin{lemma}
  \label{lemma:factorized-quadratic-growth}
  Let Assumption~\ref{assumption:low-rank-quadratic-growth} hold for
  rank $r_\star$, then for all $R\in\R^{n\times r_\star}$ such that
  $\mc{A}(RR^\top)=b$, we have
  \begin{equation*}
    \vphi_f(R) \ge \vphi_f(R_\star) + \frac{\alpha_{\vphi_f}}{2}
    \dist(R, S)^2.
  \end{equation*}
  where
  $\alpha_{\vphi_f}=2(\sqrt{2}-1)\sigma_{r_\star}^2(R_\star)\alpha_f$.
\end{lemma}
The above Lemma allows us to establish quadratic growth of $\vphi_f$
and $\vphi$
based on low-rank quadratic growth of $f$. In particular, applying
Lemma~\ref{lemma:penalized-quadratic-growth}, we get that $\vphi$ has
local quadratic growth with constant
\begin{equation*}
  \alpha_\vphi = \frac{1}{2} \cdot
  \alpha_{\vphi_f} \ge (\sqrt{2}-1)\sigma_{r_\star}(R_\star)^2\alpha_f
\end{equation*}
in a neighborhood of $R_\star$. We could then apply
Theorem~\ref{theorem:quadratic-implies-regularity} here and obtain
norm convexity and sub-differential regularity on the penalized
objective $\vphi$. We summarize this in the following theorem.
\begin{theorem}[Geometry of matrix factorization with exact rank]
  \label{theorem:geometry-exact-specified}
  Suppose $f$ is convex and satisfies
  Assumption~\ref{assumption:smooth-f},~\ref{assumption:dual-existence}, 
  Assumption~\ref{assumption:low-rank-quadratic-growth} with rank
  $r_\star$ 
  and constant $\alpha_f>0$, and
  Assumption~\ref{assumption:constraint-qualification-matrix}, then
  for sufficiently large $\lambda$, 
  $\vphi(R)=\vphi_f(R) + \lambda \vphi_g(R)$ satisfies the norm
  convexity~\eqref{equation:comparison} with constant $\ell<5$ and
  sub-differential regularity~\eqref{equation:regularity} with
  constant
  $\wt{\ell}<2(\sqrt{2}+1)\ell/(\sigma_{r_\star}(R_\star)^2\alpha_f)$.
\end{theorem}

\subsection{Matrix growth: the rank over-specified case}
\label{section:over-specification}
In many real problems, the true rank $r_\star$ cannot be known
exactly. In these cases, a common strategy is to conjecture an upper
bound $r$ on the rank and factorize $X=UU^\top$ for
$U\in\R^{n\times r}$. We show that over-specifying the
rank will preserve quadratic growth when $f$ is linear. Hence, for
solving SDPs, local linear convergence can still be achieved
when we over-specify the rank. In constrast, quadratic growth will not
be preserved with generic convex $f$.

\subsubsection{Quadratic growth is not preserved in general}
\label{section:over-specification-fails}
Recall that when converting
quadratic growth of $f$ into that of $\vphi_f$
(Lemma~\ref{lemma:factorized-quadratic-growth}), we relied on
Lemma~\ref{lemma:matrix-factorization-bound}, which
says that $\lfro{RR^\top-R_\star R_\star^\top}$ grows at least
linearly in $\lfro{R-R_\star}$:
\begin{equation*}
  \lfro{RR^\top-R_\star R_\star^\top} \ge
  \sqrt{2(\sqrt{2}-1)} \sigma_r(R_\star) \cdot \lfro{R-R_\star}.
\end{equation*}
This bound requires $\sigma_r(R_\star)>0$, so if we used an upper
bound $r>r_\star$ and factorized $X_\star=U_\star U_\star^\top$ for
$U_\star\in\R^{n\times r}$, then $\sigma_r(U_\star)=0$, and the bound
becomes vacuous. The following example demonstrates that the growth
can indeed be slower when we over-specify the rank.

\begin{example}
  \label{example:overspecified-rank}
  Let $R_\star\in\R^{n\times r_\star}$ have full column rank and
  $U_\star=[R_\star,\zeros_n]\in\R^{n\times(r_\star+1)}$, so that
  $U_\star U_\star^\top = X_\star$. Let
  $U=[R_\star,y]\in\R^{n\times(r_\star+1)}$, where $y\in\R^n$ is such
  that $R_\star^\top = \zeros$. Then,
  \begin{equation*}
    U^\top U_\star = \begin{bmatrix}
      R_\star^\top R_\star & \zeros_r \\
      y^\top R_\star & 0
    \end{bmatrix} = \begin{bmatrix}
      R_\star^\top R_\star & \zeros_r \\
      \zeros_r^\top & 0
    \end{bmatrix} \succeq 0,
  \end{equation*}
  so $U$ and $U_\star$ are optimally aligned. However, we
  have
  \begin{equation*}
    \lfro{UU^\top - U_\star U_\star^\top} =
    \lfro{R_\star R_\star^\top + yy^\top - R_\star R_\star^\top} =
    \ltwo{y}^2
  \end{equation*}
  and
  \begin{equation*}
    \lfro{U - U_\star} = \lfro{[R_\star,y] -
      [R_\star,\zeros_n]} = \ltwo{y}.
  \end{equation*}
  The distance in the PSD space depends quadratically in the distance
  in the low-rank space, not linearly.

  Taking $h(X)=\frac{1}{2}\lfro{X-X_\star}^2$ for example, the
  factorized version $\vphi(U)=h(UU^\top)$ will only have fourth-order
  growth around $U_\star$ in certain directions. The prox-linear
  algorithm will \emph{not} have local linear convergence due to this
  slow growth, for the same reason that gradient descent will not
  converge linearly on $f(x)=x^4$.
  
\end{example}

Knowing that $\lfro{UU^\top-U_\star U_\star^\top}$ can be
$O(\dist(U,U_\star)^2)$ (as opposed to linear in the distance), we
extend Lemma~\ref{lemma:matrix-factorization-bound} in the following
result, showing that a quadratic lower bound is indeed
true. Consequently, any problem~\eqref{problem:constrained-sdp} with
quadratic growth will have at least fourth-order growth under rank
over-specification.
\begin{lemma}[Matrix growth bound under rank over-specification]
  \label{lemma:overspecified-factorization-bound}
  Let $U_\star,U\in\R^{n\times r}$ be such that
  $U^\top U_\star\succeq 0$, and $\Delta=U-U_\star$. Then
  \begin{equation*}
    \begin{aligned}
      & \quad \lfro{UU^\top - U_\star U_\star^\top}^2
      \ge \frac{2(\sqrt{2}-1)}{9r}\lfro{\Delta}^4.
    \end{aligned}
  \end{equation*}
\end{lemma}
The proof can be found in
Appendix~\ref{appendix:proof-overspecified-factorization}. 

\subsubsection{Quadratic growth is preserved on linear objectives}
\label{section:over-specification-succeeds}
In contrast to the generic case,
we show that for linear objectives, rank
over-specification preserves quadratic growth under some mild
additional assumptions. Letting $f(X)=\<C,X\>$ for $C\in\symm^n$,
problem~\eqref{problem:constrained-sdp} reads
\begin{equation}
  \label{problem:sdp}
  \begin{aligned}
    \minimize & ~~ \<C,X\> \\
    \subjectto & ~~ \mc{A}(X) = b \\
    & ~~ X \succeq 0.
  \end{aligned}
\end{equation}
Problem~\eqref{problem:sdp} admits a rank-$r_\star$ solution
$X_\star=R_\star R_\star^\top$, where $R_\star\in\R^{n\times r_\star}$
has full column rank
Let $X_\star=Q_1\Sigma_x Q_1^\top$ be the eigenvalue decomposition,
where $Q_1\in\R^{n\times r_\star}$ and
$\Sigma_x\in\R^{r_\star\times r_\star}$ is diagonal. We can then take
$R_\star=Q_1\sqrt{\Sigma_x}$. Let $Q_2\in\R^{n\times(n-r_\star)}$ be
the orthogonal complement of $Q_1$, i.e. such that
$[Q_1, Q_2]\in\R^{n\times n}$ is an orthogonal matrix.

We will assume an additional assumption on the dual SDP
(cf.~\eqref{equation:dual-optimum}), which is
a mild condition that guarantees a low
rank solution of an SDP is unique~\cite{AlizadehHaOv97}.

\begin{assumption}[Strict complementarity and dual non-degeneracy]
  \label{assumption:dn}
  There exists a pair of dual optimal $(y_\star,Z_\star)$ such that
  the $(r_\star+1)$-th smallest singular value of $Z_\star$ is lower
  bounded by $\eigvaldual>0$. There exists some constant $\sigmadn>0$
  such that
  \begin{equation*}
    \sum_{i\in[k]} \<Q_1^\top A_iQ_1, W\>^2 \ge \sigmadn \cdot
    \lfro{W}^2,~\text{for all}~W\in\symm^{r_\star}.
  \end{equation*}
\end{assumption}
Assuming this condition, we can lower bound the growth of $f$ as
follows. For any $X\succeq 0$, $\mc{A}(X)=b$, we have
$\<\mc{A}^*(y_\star),X-X_\star\>=\<y_\star,\mc{A}(X)-\mc{A}(X_\star)\>=0$. Hence
\begin{equation}
  \label{equation:linear-growth}
  \begin{aligned}
    & \quad \<C,X-X_\star\> = \<\mc{A}^*(y_\star)+C,X-X_\star\> \\
    & = \<Z_\star,X-X_\star\> = 
    \<Z_\star,X\> \ge \eigvaldual\lnuc{Q_2^\top XQ_2}.
  \end{aligned}
\end{equation}
The following result further lower bounds $\lnuc{Q_2^\top XQ_2}$ by
$\dist(U, U_\star)^2$, and thereby showing the quadratic growth of
$\vphi_f(U)$. See Appendix~\ref{appendix:proof-linear-overspecified}
for the proof.
\begin{lemma}
  \label{lemma:overspecified-growth}
  Let $r\ge r_\star$, $X_\star=U_\star U_\star^\top$ and $X=UU^\top$ for
  $U,U_\star\in\R^{n\times r}$.
  Suppose Assumption~\ref{assumption:dn} holds, and the feasible set
  \begin{equation*}
    \mc{M}_r \defeq \set{U\in\R^{n\times r}:\mc{A}(UU^\top)=b}
  \end{equation*}
  is a smooth manifold,
  then there exists a neighborhood of $U_\star$ and a positive
  constant $c>0$ such that for all feasible $X=UU^\top$ in this
  neighborhood, 
  \begin{equation*}
    \lnuc{Q_2^\top XQ_2} \ge c \cdot \dist(U,U_\star)^2,
  \end{equation*}
  where the constant $c$ is
  \begin{equation*}
    c = \frac{1}{2\kappa+1},~~~ \kappa \defeq
    \frac{4\lambda_1(X_\star)}{\lambda_{r_\star}(X_\star)} \cdot 
    \frac{\sum_{i=1}^k \lfro{Q_2^\top A_iQ_1}^2}{\sigmadn}.
  \end{equation*}
  Consequently, $\vphi_f$ have local quadratic growth around $U_\star$
  with constant $\alpha_{\vphi_f}=2c\eigvaldual$. 
\end{lemma}

By the above Lemma, we establish quadratic growth of the rank
over-specified objective $\<C,UU^\top\>$ and thereby verifying
Assumption~\ref{assumption:low-rank-quadratic-growth} for the
constrained problem. As a direct consequence, we obtain norm convexity
and sub-differential regularity from
Theorem~\ref{theorem:quadratic-implies-regularity}. We summarize this
in the following theorem.
\begin{theorem}[Geometry of factorized SDPs with rank
  over-specification]
  \label{theorem:linear-objective-over-specified}
  Suppose $f(X)=\<C,X\>$ and
  Assumptions~\ref{assumption:constraint-qualification-matrix}
  and~\ref{assumption:dn} hold for rank $r$, then the solution to
  problem~\eqref{problem:sdp} is unique, and
  $\vphi_f(U)=\<C,UU^\top\>$ has local quadratic growth: for all $U$
  near $U_\star$ and feasible,
  \begin{equation*}
    \vphi_f(U) - \vphi_f(U_\star) \ge \frac{\alpha_{\vphi_f}}{2}
    \cdot \dist(U,U_\star)^2,
  \end{equation*}
  where $\alpha_{\vphi_f}>0$ depends on
  $(X_\star,\mc{A},\lambda,\sigmadn)$ but not $r$ (so that
  Assumption~\ref{assumption:low-rank-quadratic-growth} holds).
  Further, for sufficiently
  large $\lambda$,
  $\vphi(U)=\vphi_f(U) + \lambda \vphi_g(U)$ satisfies the norm
  convexity~\eqref{equation:comparison} and sub-differential
  regularity~\eqref{equation:regularity} (with $R$ replaced by $U$).
\end{theorem}
Theorem~\ref{theorem:linear-objective-over-specified} applies
generally to the Burer-Monteiro factorization of SDPs, 
and spells out the reason why over-specifying the rank often works in
practice -- quadratic growth is carried to from $f(X)=\<C,X\>$ to
$\vphi_f(U)=\<C,UU^\top\>$, as long as dual non-degeneracy holds.

\subsection{Algorithmic consequences}
We now adapt Theorem~\ref{theorem:linear-convergence} in both the
exact rank case or the rank over-specified case with linear objectives
to obtain local linear convergence of the prox-linear algorithm.
\begin{theorem}[Local linear convergence of
  factorized semidefinite optimization]
  \label{theorem:sdp-linear-convergence}
  Under the settings of Theorem~\ref{theorem:geometry-exact-specified}
  or~\ref{theorem:linear-objective-over-specified}, let
  $\alpha_{\vphi}$ be the local quadratic growth constant of $\vphi$.
  Initializing sufficiently close to the minimizing set $S$, the
  prox-linear algorithm converges linearly:
  \begin{equation*}
    \vphi(R^{k+1}) - \vphi_\star \le q(\vphi(R^k) - \vphi_\star),
  \end{equation*}
  where
  \begin{equation*}
    q \defeq 1 -
    \frac{1}{209+5000(L_f+\lambda\opnorm{\mc{A}})/\alpha_\vphi}.
  \end{equation*}
  If we initialize in a sufficiently small neighborhood of $S$ and let
  $\lambda=\Lambda$ be the lowest possible choice as provided in
  Theorem~\ref{theorem:quadratic-implies-regularity}, then the linear
  rate is $q=1-M^{-1}$ with
  \begin{equation*}
    M = O\left( \frac{\lipf + \opnorm{\mc{A}}
        \ltwo{y_\star}\sigmamax/\sigmacq}{\alpha_\vphi} \right).
  \end{equation*}
\end{theorem}
The proof as well as discussions on this last linear rate can be
found in Appendix~\ref{appendix:proof-corollary-rate}.
\section{Examples of quadratic growth}
\label{section:examples}
In this section, we provide examples of
problem~\eqref{problem:constrained-sdp} that have low-rank quadratic
growth, i.e. satisfying
Assumption~\ref{assumption:low-rank-quadratic-growth}. By giving
conditions under which these are true, we identify some situations in
which the geometric results given in
Theorems~\ref{theorem:geometry-exact-specified}
or~\ref{theorem:linear-objective-over-specified} will hold.

\subsection{Linear objectives}
As we see in Theorem~\ref{theorem:linear-objective-over-specified},
our sufficient conditions for quadratic growth requires checking CQ,
strict complementarity, and dual non-degeneracy. We illustrate showing
these conditions for the SDP for $\Z_2$ synchronization and SO(d)
synchronization when the data contains strong signal.

\begin{example}[$\Z_2$ synchronization]
  Let $x_\star\in\set{\pm 1}^n$ be an unknown binary vector. The
  $\Z_2$-synchronization problem is to recover $x_\star$ from the
  matrix of noisy observations
  \begin{equation*}
    A = \frac{\lambda}{n} x_\star x_\star^\top + W,
  \end{equation*}
  where $W$ is a Gaussian Orthogonal Ensemble (GOE): $W=W^\top$,
  $W_{ij}\sim\normal(0,1/n)$ for $i<j$, $W_{ii}\sim\normal(0,2/n)$,
  and these entries are independent. This problem is a simplified
  model for community detection problems such as in the stochastic
  block model~\cite{BandeiraBoVo16}.

  The maximum likelihood estimate of the above problem is
  computationally intractable for its need to search over $2^n$
  possibilities. However, the maximum likelihood problem can be
  relaxed into an SDP: letting $C=-A$, we solve
  \begin{equation}
    \label{problem:maxcut}
    \begin{aligned}
      \minimize & ~~ \<C,X\> \\
      \subjectto & ~~ \diag(X) = \ones \\
      & ~~ X \succeq 0.
    \end{aligned}
  \end{equation}
  We are interested in when the relaxation is tight, or, that $x_\star
  x_\star^\top$ is the unique solution
  to~\eqref{problem:maxcut}. Recent work~\cite{BandeiraBoSi17}
  establishes that when
  \begin{equation*}
    \lambda>\sqrt{(2+\eps)\log n},
  \end{equation*}
  with high probability, $x_\star x_\star^\top$ is the unique
  solution to~\eqref{problem:maxcut} and strict complementarity
  holds. If this happpens, dual non-degeneracy also holds: we have
  $X_\star = x_\star x_\star^\top$, $Q_1=x_\star/\sqrt{n}$, and the
  matrices $Q_1^\top A_iQ_1 = [x_\star]_i^2=1$ so they certainly span
  $\symm^1=\R$. 
  
  Finally, we note that CQ holds for \emph{any} MaxCut problem: the
  constraints are $A_i=e_ie_i^\top$ and $b_i=1$ for
  $i\in[n]$. For $X_\star=R_\star R_\star^\top$, constraint
  qualification requires that $\set{e_1e_1^\top
    R_\star,\dots,e_ne_n^\top R_\star}$ are
  linearly independent, or that $R_\star$ have non-zero rows. This has
  to be true, as the rows have norm one.

  Putting together, the assumptions of
  Theorem~\ref{theorem:linear-objective-over-specified} will hold, and
  the factorized problem with rank $r\ge 1$ will have quadratic
  growth, norm convexity, and sub-differential regularity.
\end{example}

\begin{example}[SO(d) synchronization]
  The SO(d) synchronization problem is an multi-dimensional extension
  of the MaxCut problem: we are interested in
  recovering $n$ orthogonal matrices
  $Q_1,\dots,Q_n\in\R^{d\times d}$ 
  given their noisy pairwise compositions
  \begin{equation*}
    A_{ij} = Q_iQ_j^\top + {\rm noise} \in \R^{d\times d}.
  \end{equation*}
  Arranging $A_{ij}$ into $A\in\R^{nd\times nd}$ and forming the
  decision variable $R\in\R^{(nd)\times d}$ with row blocks
  $R_i\in\R^{d\times d}$, we solve (for $C=-A$)
  \begin{align*}
    \minimize & ~~\sum_{i,j} \<C_{ij}, R_iR_j^\top\> = \<C,RR^\top\>
    \\
    \subjectto & ~~ R_iR_i^\top = I_{d\times d}.
  \end{align*}
  The SDP relaxation is
  \begin{equation}
    \label{problem:sod}
    \begin{aligned}
      \minimize & ~~ \<C,X\> \\
      \subjectto & ~~ X_{((i-1)d+1):id,((i-1)d+1):id} = I_{d\times d} \\
      & ~~ X \succeq 0.
    \end{aligned}
  \end{equation}
  By symmetry, there are $k=nd(d+1)/2$ equality constraints specifying
  the diagonal blocks.

  CQ holds for all SDPs of the form~\eqref{problem:sod}. Indeed, let
  $X_\star=R_\star R_\star^\top$ with
  $R_\star\in\R^{nd\times r_\star}$ be the low-rank solution to the
  above SDP and $R_i\in\R^{d\times r_\star}$ be the $i$-the row
  block. Define $R^0_{i,j}$ as the matrix only keeping the $j$-th row
  of $R_i$; $R^0_{i,jk}$ as the matrix swapping the $j$-th and $k$-th
  row of $R_i$ and all the other rows set to zero. Constraint
  qualification requires that for all $i\in[n]$, the matrices
  \begin{equation*}
    \set{R^0_{i,j}:j\in[n]} \cup \set{R^0_{i,jk}:j,k\in[n],j\neq k}
  \end{equation*}
  are linearly independent. A sufficient condition is that for all
  $i$, the rows of $R_i$ are linearly independent. Again, as
  $R_iR_i^\top=I_{d\times d}$, the rows of $R_i$ must be orthonormal
  and so linearly independent.

  When the noise is small enough, we may expect that the signal
  dominates the noise, strict complementarity holds, and
  $\rank(X_\star)=d$ with $R_\star$ having orthogonal row blocks that
  are close to 
  $Q_i$. When this happens, we claim that the dual
  non-degeneracy holds. Indeed, as the row blocks are orthogonal,
  $R_\star/\sqrt{n}$ has orthonormal columns, and so dual
  non-degeneracy requires that the set
  \begin{equation*}
    \set{\frac{1}{n}r_{(i-1)d+j} r_{(i-1)d+j}^\top,~
    \frac{1}{n}\left(r_{(i-1)d+j_1} r_{(i-1)d+j_2}^{\top} 
      + r_{(i-1)d+j_2} r_{(i-1)d+j_1}^{\top}\right)
    :i\in[n],j,j_1\neq j_2\in[d]}
  \end{equation*}
  spans $\symm^d$ ($r_i$ is the $i$-th row of $R_\star$). As
  $\set{r_1,r_2,r_3}$ forms an orthonormal basis of $\R^d$, the subset
  with $i=1$ spans $\symm^d$ already.

  Showing strict complementarity under strong signal requires
  concentration results similar to the $\Z_2$ synchronization case,
  which is not the focus of the present paper.
\end{example}

\subsection{Quadratic objectives}
We briefly illustrate how to show quadratic growth in objectives with
more natural quadratic behavior.

\begin{example}[Low-rank matrix sensing]
  Let $X_\star=R_\star R_\star^\top$ be a rank-$r_\star$ matrix, where
  $R_\star$ has norm-one rows. Let
  $c_i\in\R^n,i\in[N]$ be i.i.d. Gaussian random vectors,
  $C_i=c_ic_i^\top$,
  and suppose we observe
  \begin{equation*}
    d_i = c_i^\top X_\star c_i = \<X_\star, c_ic_i^\top\> = \<C_i,
    X_\star\>.
  \end{equation*}
  The goal is to recover $X_\star$. This is a binary phase
  retrieval problem when $r_\star=1$, and is in general a low-rank
  matrix sensing problem with some additional norm constraints.
  
  We solve
  \begin{align*}
    \minimize & ~~ \frac{1}{2N}\ltwo{\mc{C}(X) - d}^2 =
                \frac{1}{2N}\ltwo{\mc{C}(X-X_\star)}^2 \\
    \subjectto & ~~ \diag(X) = \ones,~~X \succeq 0.
  \end{align*}
  with $A_i=e_ie_i^\top$, $b_i=1$ for $i\in[n]$. Clearly, $X_\star$ is
  a solution (as the objective equals zero) and constraint
  qualification holds at $X_\star$. For any $X$ with $\rank(X) \le r$,
  where $r\ge r_\star$, $X-X_\star$ has rank at most $r_\star+r\le
  2r$, and so 
  \begin{equation*}
    \lnuc{X-X_\star}^2 \le 2r\lfro{X-X_\star}^2.
  \end{equation*}
  Standard matrix concentration results~\citep[Theorem
  10.2]{Wainwright19} show that as long as $N\ge C_0\cdot nr$, with
  high probability
  \begin{equation*}
    \frac{1}{2N}\ltwo{\mc{C}(X-X_\star)}^2 \ge
    \frac{1}{4}\lfro{X-X_\star}^2
  \end{equation*}
  uniformly over all $X$ with $\rank(X)\le r$. This gives quadratic
  growth over rank-$r$ matrices, which in turn implies nice geometries
  by Theorem~\ref{theorem:quadratic-implies-regularity}. Further, it
  is possible to achieve a lower constant $C_0$ than implied in the
  original concentration, as we only need to look at $X$ satisfying
  $X_{ii}=1$, which as a set will have lower metric entropy.
\end{example}
\section{Examples and numerical experiments}
Our results in Section~\ref{section:matrix-growth} are applicable on a
variety of low-rank semidefinite optimization problems, including
\begin{itemize}
\item $\Z_2$ synchronization (MaxCut SDP) and ${\rm SO}(d)$
  synchronization;
\item Low-rank matrix sensing;
\item Random quadratics problem,
\end{itemize}
whose details are deferred to Appendix~\ref{section:examples}
and~\ref{section:experiments}. Numerical experiments on synthetic MaxCut SDPs
and random quadratics problems are presented in
Appendix~\ref{section:experiments}.

\section{Conclusion}
We considered a family of constrained composite optimization problems
and proposed solving them by the prox-linear algorithm on the exact
penalty formulation. We established its local linear convergence assuming
quadratic growth and build a matrix-specific theory for showing such
quadratic growth in various types of semidefinite optimization
problems. For future work, it would be of interest to generalize such
convergence theory to composite problems beyond exact penalty
functions. It would be also valuable to implement our algorithm as a
scalable low-rank SDP solver.

\bibliographystyle{abbrvnat}
\bibliography{main}

\appendix
\section{Additional notation}
In Euclidean spaces, $\proj_A$ denote the orthogonal projection onto a
set $A$. 
For a smooth manifold $\mc{M}\subset \R^n$, embedded in $\R^n$, let
$\mc{T}_x\mc{M}$ be its tangent space at $x$.

For function $f$ on a smooth manifold $\mc{M}\subset \R^n$, let
$\rgrad f$ and $\rhess f$ denote its Riemannian gradient and Hessian
operators, whose Euclidean representations are defined as
(cf.~\cite{AbsilMaSe09})
\begin{equation}
  \label{equation:riemannian-grad-hess}
  \begin{aligned}
    & \rgrad f(x) = \proj_{\mc{T}_x\mc{M}}\grad f(x)~~~{\rm and}
    ~~~\rhess f(x)[u] = \proj_{\mc{T}_x\mc{M}}(D\rgrad
    f(x)[u]),~\forall u\in\mc{T}_x\mc{M}.
  \end{aligned}
\end{equation}

\section{Technical lemmas}
\subsection{Geometries of Riemannian manifolds}
We frequently use the orthogonal projection onto a Riemannian manifold
$\mc{M}$ in our proofs. For any closed set $M\in\R^n$, the orthogonal
projection onto $M$ is the set
\begin{equation*}
  \proj_M(x) \defeq \argmin_{y\in M} \ltwo{x-y}.
\end{equation*}
When $M$ is convex, the projection $\proj_M$ exists and is unique for
all $x\in\R^n$. For
smooth manifolds, which are non-convex in general, projections are
still well-defined locally. This is stated in the following lemma (cf.
\begin{lemma}[Lemma 4,~\cite{LewisMa08}]
  \label{lemma:projection-exists}
  Let $\mc{M}\subset\R^n$ be a smooth manifold and
  $x\in\mc{M}$, then there exists a neighborhood $\mc{X}$
  of $x$ such that $\proj_{\mc{M}}$ exists and is unique in
  $\mc{X}$. 
\end{lemma}
While Lemma~\ref{lemma:projection-exists}
only gives guarantees around a single
point $x\in\mc{M}$, we note that it applies to any compact
subset $S\subset\mc{M}$. Indeed, for each $x\in S$, there exists a
neighborhood $\ball(x,\eps_x)$ such that the projection uniquely
exists. Now, as the union of $\ball(x,\eps_x)$ covers
$S$ (take the open balls), by compactness, there exists a finite
sub-cover, i.e. $\set{x_1,\dots,x_k}\in S$ such that
$S\subset \bigcup_i \ball(x_i,\eps_{x_i})$. Taking
$\eps_0=\min_{i\in[k]}\eps_{x_i}$, the neighborhood $\mc{N}(S,\eps_0)$
is a desired neighborhood of $S$ in which the projection uniquely
exists.

For any $y$ with a well-defined projection, we have the orthogonality
condition
\begin{equation*}
  \<y - \proj_{\mc{M}}(y), v\> = 0~~{\rm
    for~all}~~v\in\mc{T}_{\proj_{\mc{M}}(y)}\mc{M},
\end{equation*}
or simply $y-\proj_{\mc{M}}(y)\in\mc{N}_{\proj_{\mc{M}(y)}}$. Hence,
$\proj_{\mc{M}}(y)$ is also the projection of $y$ onto the tangent
space $\proj_{\mc{M}}(y)+\mc{T}_{\proj_{\mc{M}}(y)}\mc{M}$. Building
on this, we show an approximate Pythagorean identity for repeated
projections onto smooth manifolds. 
\begin{lemma}
  \label{lemma:projection-bound}
  Let $\mc{M}\subset\R^n$ be a smooth manifold and $S\subset\mc{M}$ be
  a compact smooth submanifold of $\mc{M}$. For
  any $\delta\in(0,1)$, there exists a neighborhood $\mc{N}(S,\eps)$
  such that for all $x\in\mc{N}(S,\eps)$, we have
  \begin{equation*}
    \ltwo{x-\proj_{\mc{M}}(x)}^2 +
    \ltwo{\proj_{\mc{M}}(x)-\proj_{S}\proj_{\mc{M}}(x)}^2 \ge
    (1-\delta)\ltwo{x-\proj_{S}\proj_{\mc{M}}(x)}^2.
  \end{equation*}
  Consequently,
  \begin{equation*}
    \ltwo{x-\proj_{\mc{M}}(x)}^2 + \dist(\proj_{\mc{M}}(x), S)^2 \ge
    (1-\delta)\dist(x,S)^2.
  \end{equation*}
\end{lemma}
\begin{proof}
  Let $\delta\in(0,1)$ be arbitrary.
  Let $\wt{x}_\star\in S$ and consider a point
  $x\in\ball(\wt{x}_\star,\eps)$, with $\eps$ to be determined. Let
  $x_\star=\proj_S(x)$, i.e.
  \begin{equation*}
    \ltwo{x-x_\star} = \inf_{x'\in S} \ltwo{x-x'} = \dist(x,S) \le \eps.
  \end{equation*}
  Let $y=\proj_{\mc{M}}(x)$ and $z=\proj_{S}(y)$. By
  Lemma~\ref{lemma:projection-exists}, when $\eps$ is small
  enough, $y$ will be well-defined and satisfies
  \begin{equation*}
    \ltwo{y-\wt{x}_\star} \le \ltwo{y-x} + \ltwo{x-\wt{x}_\star} \le
    2\ltwo{x-\wt{x}_\star} \le 2\eps.
  \end{equation*}
  So for $\eps$ sufficently small, $y$ is also in a small neighborhood
  of $\wt{x}_\star$ and thus $z$, the projection of $y$ onto $S$, 
  is also well-defined and
  \begin{equation*}
    \ltwo{z-\wt{x}_\star} \le \ltwo{z-y} + \ltwo{y-\wt{x}_\star} \le
    2\ltwo{y-\wt{x}_\star} \le 4\ltwo{x-\wt{x}_\star} \le 4\eps.
  \end{equation*}

  Now, consider the tangent space of $\mc{M}$ at $y$, and let
  $z'=y+\proj_{\mc{T}_y\mc{M}}(z-y)$ be the projection of $z$ onto the
  affine tangent space at $y$. By the Pythagorean identity for
  projection onto linear subspace, we have
  \begin{equation*}
    \ltwo{x-z'}^2 = \ltwo{x-y}^2 + \ltwo{y-z'}^2 \le \ltwo{x-y}^2 +
    \ltwo{y-z}^2.
  \end{equation*}
  It remains to connect $\ltwo{x-z'}$ and $\ltwo{x-z}$.
  As $y$ is close to $\wt{x}_\star$ and
  the manifold is smooth, the tangent space $y+\mc{T}_y\mc{M}$ is an
  accurate approximation of $\mc{M}$ at $y$, in the sense that
  (cf.~\cite{AbsilMaSe09})
  \begin{equation*}
    \ltwo{z-z'} = o(\ltwo{z-y})~~~{\rm as}~\eps\to 0.
  \end{equation*}
  Choosing $\eps$ sufficiently small, we have
  \begin{equation*}
    \ltwo{z-z'} \le \frac{\delta}{2}\ltwo{z-y} \le
    \frac{\delta}{2}\ltwo{x_\star-y} \le 
    \frac{\delta}{2}(\ltwo{x_\star-x}+\ltwo{x-y}) \le
    \delta\ltwo{x-x_\star} \le \delta\ltwo{x-z},
  \end{equation*}
  which implies that
  \begin{equation*}
    \ltwo{x-z'} \ge \ltwo{x-z} - \ltwo{z'-z} \ge
    (1-\delta)\ltwo{x-z}.
  \end{equation*}
  Putting together, we get
  \begin{equation*}
    (1-\delta)^2\ltwo{x-z}^2 \le \ltwo{x-z'}^2 \le \ltwo{x-y}^2 +
    \ltwo{y-z}^2.
  \end{equation*}
  This is the desired bound, and we have shown that it holds for all
  $x\in\ball(\wt{x}_\star,\eps)$ for sufficiently small $\eps$. Now,
  for each $\wt{x}_\star\in\mc{S}$ we can establish such a
  neighborhood, and using the finite sub-cover property, we can find a
  finite set of minima each associated with a neighborhood. Choosing
  the minimum $\eps$ of these finitely many neighborhood sizes, the
  desired property holds for all $x\in\mc{N}(S,\eps)$.
  

\end{proof}

\subsection{Sub-differential and growth of constraints}
The following lemma is useful in proving our main geometric result. It
looks at the behavior of the penalty term $\ltwo{Ac(x)-b}$.
\begin{lemma}
  \label{lemma:norm-bound}
  Let Assumption~\ref{assumption:smoothness}
  and~\ref{assumption:constraint-non-degeneracy} hold, then
  for any $\delta\in(0,1)$ there exists a neighborhood
  $\mc{N}(S,\eps)$
  with $\eps\le\delta\sigmacq/((1+\delta)\opnorm{A}\smoothc)$
  such that the following holds: for any $x\in\mc{N}(S,\eps)$ with
  $Ac(x)-b\neq 0$, 
  \begin{enumerate}[(1)]
  \item Let $y=\proj_{\mc{M}}(x)$ be the projection of $x$ onto
    $\mc{M}$, then
    \begin{equation*}
      \ltwo{Ac(x)-b} \ge (1-\delta)\sigmacq\ltwo{x-y}.
    \end{equation*}
  \item We have
    \begin{equation*}
      \ltwo{\grad\vphi_g(x)} \cdot \dist(x,S) \ge
      (1-\delta)\ltwo{Ac(x)-b} = (1-\delta)\vphi_g(x).
    \end{equation*}
  \end{enumerate}
\end{lemma}
\begin{proof}
  \begin{enumerate}[(1)]
  \item 
    As $\mc{M}$ is smooth, by Lemma~\ref{lemma:projection-exists}, the
    projection uniquely exists for all $x\in\mc{N}(S,\eps)$ for some
    small $\eps>0$. Let $y=\proj_{\mc{M}}(x)$, so the triangle
    inequality implies $\dist(y,S)\le \dist(x,y)+\dist(x,S)\le
    2\dist(x,S)$. 

    We can now perform a Taylor expansion and get
    \begin{equation*}
      \ltwo{Ac(x)-b} = \ltwo{Ac(x)-Ac(y)} = \ltwo{A(\grad
        c(y)^\top(x-y)+r)} \ge \ltwo{A\grad c(y)^\top(x-y)} -
      \ltwo{Ar}, 
    \end{equation*}
    where by smoothness $\ltwo{r}\le\smoothc\ltwo{x-y}^2$. Now, as $y$
    is the projection of $x$ onto $\mc{M}$, $x-y$ must be orthogonal to
    the tangent space $\mc{T}_y\mc{M}$, so we have
    $x-y\in\mc{N}_y\mc{M}={\rm span}\set{\grad
      c(y)a_1,\dots, \grad c(y)a_k}$. By
    Assumption~\ref{assumption:constraint-non-degeneracy}, whenever 
    $\dist(x,S)\le \eps_0/2$, we have $\dist(y,S)\le\eps_0$, so $y$ is
    in the neighborhood of constraint qualification and thus
    \begin{equation*}
      \ltwo{A\grad c(y)^\top(x-y)} \ge \sigma_{\min}(A\grad
      c(y)^\top)\ltwo{x-y} \ge \sigmacq \ltwo{x-y}.
    \end{equation*}
    This gives us
    \begin{equation*}
      \ltwo{Ac(x)-b} \ge \sigmacq\ltwo{x-y} -
      \opnorm{A}\smoothc\ltwo{x-y}^2 \ge
      (1-\delta)\sigmacq\ltwo{x-y}
    \end{equation*}
    as long as $\ltwo{x-y}\le\delta\sigmacq/(\opnorm{A}\smoothc)$,
    which is satisfied if
    $\dist(x,S)\le\delta\sigmacq/(\opnorm{A}\smoothc)$. 
  \item Observing that $\dist(x,S)\ge\ltwo{x-y}$, it suffices to
    show the result with $\dist(x,S)$ replaced by $\ltwo{x-y}$. Recall
    that for $x\notin\mc{M}$,
    \begin{equation*}
      \grad\vphi_g(x) = \frac{\grad c(x) A^\top(Ac(x)-b)}{\ltwo{Ac(x)-b}},
    \end{equation*}
    so it suffices to show that
    \begin{equation*}
      \ltwo{\grad c(x) A^\top(Ac(x)-b)} \cdot \ltwo{x-y} \ge
      (1-\delta)\ltwo{Ac(x)-b}^2.
    \end{equation*}
    We have
    \begin{align*}
      & \quad \ltwo{\grad c(x) A^\top(Ac(x)-b)} \cdot \ltwo{x-y}
      \\
      & \ge \<\grad c(x) A^\top(Ac(x)-b), x-y\> = \<Ac(x)-b,
        A\grad c(x)^\top(x-y)\> \\
      & = \ltwo{Ac(x)-b}^2 + \<Ac(x)-b,r\>,
    \end{align*}
    where $r = Ac(y)-Ac(x) - A\grad c(x)^\top(y-x)$ satisfies
    $\ltwo{r}\le \opnorm{A}\smoothc\ltwo{x-y}^2$. Similar to (1), the bound
    \begin{equation}
      \label{equation:constraint-higher-order}
      \ltwo{r} \le \opnorm{A}\smoothc\ltwo{x-y}^2 \le \delta\ltwo{Ac(x)-b}
    \end{equation}
    holds when
    \begin{equation*}
      \delta(\sigmacq\ltwo{x-y} - \opnorm{A}\smoothc\ltwo{x-y}^2) \ge
      \opnorm{A}\smoothc\ltwo{x-y}^2,
    \end{equation*}
    or that
    \begin{equation*}
      \ltwo{x-y} \le \frac{\delta\gamma}{(1+\delta)\opnorm{A}\smoothc}.
    \end{equation*}
    Substituting~\eqref{equation:constraint-higher-order} into
    the main inequality gives
    \begin{equation*}
      \<\grad c(x) A^\top(Ac(x)-b), x-y\> \ge \ltwo{Ac(x)-b}
      \left( \ltwo{Ac(x)-b} - \ltwo{r} \right) \ge (1-\delta)\ltwo{Ac(x)-b}^2,
    \end{equation*}
    the desired bound. We note that both (1) and (2) hold if
    $\dist(x,S)\le\delta\sigmacq/[(1+\delta)\opnorm{A}\smoothc]$. 
  \end{enumerate}
\end{proof}

\subsection{Global growth of $h(X)$}
\begin{lemma}
  \label{lemma:convex-growth}
  Let Assumption~\ref{assumption:quadratic-growth} hold and suppose
  $\lambda>\ltwo{y_\star}$. For any $\eps>0$, there exists some
  $\delta(\eps)>0$ such that if $X\succeq 0$ and
  $h(X)-h_\star\le\delta(\eps)$, then
  $\lfro{X-X_\star}\le\eps$. Further, $\delta(\eps)\to 0$ as $\eps\to
  0$. 
\end{lemma}
\begin{proof}
  This is a direct consequence of convexity. Define
  \begin{equation*}
    \delta(\eps) = \min_{X\succeq 0, \lfro{X-X_\star}=\eps} h(X) - h_\star. 
  \end{equation*}
  As $h_\star$ is the unique global minimum, by compactness, we have
  $\delta(\eps)>0$. Now, take any $X\succeq 0$ with
  $h(X)-h_\star\le\delta(\eps)$, and suppose
  $\lfro{X-X_\star}>\eps$. Consider
  \begin{equation*}
    X_\eps = tX + (1-t)X_\star,~~~t = \frac{\eps}{\lfro{X-X_\star}}.
  \end{equation*}
  It is easy to verify that $\lfro{X_\eps-X_\star}=\eps$ and
  $X_\eps\succeq 0$ by convexity of the PSD cone. We then have
  $h(X_\eps)-h_\star\ge\delta(\eps)$ by our definition of $\delta$. As
  $h$ is convex, we have
  \begin{equation*}
    \delta(\eps) \le h(X_\eps) - h_\star \le th(X) + (1-t)h_\star -
    h_\star = t(h(X)-h_\star) \le \frac{\eps}{\lfro{X-X_\star}} \cdot
    \delta(\eps) < \delta(\eps),
  \end{equation*}
  a contradiction. So we must have $\lfro{X-X_\star}\le\eps$. The fact
  that $\delta(\eps)\to 0$ as $\eps\to 0$ follows by the continuity of
  $h$.
\end{proof}
\section{Proofs for Section~\ref{section:geometry}}
\subsection{Proof of Lemma~\ref{lemma:penalized-quadratic-growth}}
\label{appendix:proof-penalize-growth}
Let $x$ be sufficiently close to $S$ such that its projection
$y=\proj_{\mc{M}}(x)$ is well-defined, as guaranteed by
Lemma~\ref{lemma:projection-exists}.
As $y$ is the projection, $y$ is close to $S$ when $x$ is: we have
\begin{equation*}
  \dist(y,S) \le \dist(x,S) + \dist(x,y) \le 2\dist(x,S).
\end{equation*}
As long as $x\in\mc{N}(S,\eps_0/2)$, we have $\dist(y,S)\le\eps_0$,
thus by the quadratic growth assumption
\begin{equation}
  \label{equation:feasible-bound}
  \vphi(y) - \vphi_\star = \vphi_f(y) - \vphi_\star \ge
  \frac{\alpha_{\vphi_f}}{2} \dist(y, S)^2. 
\end{equation}
We now show that the difference $\vphi(x)-\vphi(y)$ is dominated by
the growth of the penalty term. We have
\begin{equation}
  \label{equation:normal-direction}
  \vphi(x) - \vphi(y) = \vphi_f(x) - \vphi_f(y) +
  \lambda\ltwo{Ac(x)-b}.
\end{equation}
As $\vphi_f$ is locally Lipschitz, we have
$\vphi_f(x)-\vphi_f(y)\ge -\lipphif\ltwo{x-y}$.
Applying Lemma~\ref{lemma:norm-bound}(1), for any $\delta\in(0,1)$, in
a neighborhood $\mc{N}(S,\eps_\delta)$, the penalty term can be lower
bounded as
\begin{align*}
  & \quad \lambda\ltwo{Ac(x)-b} \ge
    \lambda(1-\delta)\sigmacq\ltwo{x-y}.
\end{align*}
(Recall that $\sigmacq$ is the constraint qualification constant.)
Substituting into~\eqref{equation:normal-direction}, we get that when
$\lambda\ge \lipphif/(\delta\sigmacq)\defeq\Lambdaqg$,
\begin{equation*}
  \vphi(x) - \vphi(y) \ge \left( \lambda(1-\delta)\sigmacq - \lipphif
  \right)\ltwo{x-y} \ge \lambda(1-2\delta)\sigmacq\ltwo{x-y}.
\end{equation*}
When
\begin{equation*}
  \ltwo{x-y}\le \frac{2(1-2\delta)\lipphif}{\delta} \le
  \frac{2\lambda(1-2\delta)\sigmacq}{\alpha_{\vphi_f}},
\end{equation*}
we get 
\begin{equation}
  \label{equation:normal-bound}
  \vphi(x) - \vphi(y) \ge \frac{\alpha_{\vphi_f}}{2}\ltwo{x-y}^2.
\end{equation}

Putting together~\eqref{equation:feasible-bound}
and~\eqref{equation:normal-bound} and applying
Lemma~\ref{lemma:projection-bound} gives us
\begin{align*}
  \vphi(x) - \vphi_\star
  & = \vphi(x) - \vphi(y) + \vphi(y) - \vphi_\star \\
  & \ge \frac{\alpha_{\vphi_f}}{2} \left(\ltwo{x-y}^2 +
    \dist(y,S)^2\right) \\
  & \ge \frac{(1-\delta)\alpha_{\vphi_f}}{2}
    \dist(x,S)^2
\end{align*}
in a neighborhood of $S$ (that depends on $\delta$ but not $\lambda$).

\subsection{Proof of
  Theorem~\ref{theorem:quadratic-implies-regularity}}
\label{appendix:proof-main}
We first show that norm convexity~\eqref{equation:comparison}
implies sub-differential regularity~\eqref{equation:regularity}. Let
\begin{equation*}
  \lambda\ge\Lambdaqg = \frac{\lipphif}{\delta\sigmacq},
\end{equation*}
then by Assumption~\ref{assumption:quadratic-growth} and
Lemma~\ref{lemma:penalized-quadratic-growth}, locally we have
\begin{equation*}
  \vphi(x) - \vphi_\star \ge \frac{\alpha_\vphi}{2}\dist(x, S)^2 =
  \frac{(1-\delta)\alpha_{\vphi_f}}{2}\dist(x,S)^2.
\end{equation*}
In particular, we could take $\delta=1/2$ and
$\Lambda=2\lipphif/\sigmacq$. Plugging the above bound
into~\eqref{equation:comparison}, we get
\begin{equation*}
  \frac{\alpha_{\vphi_f}}{4}\dist(x,S)^2 \le \vphi(x) - \vphi_\star \le
  \ell\dist(0,\partial\vphi(x)) \cdot \dist(x,S),
\end{equation*}
so that
\begin{equation*}
  \dist(x,S) \le \frac{4\ell}{\alpha_{\vphi_f}}\dist(0, \partial\vphi(x)).
\end{equation*}
This shows~\eqref{equation:regularity}.

The rest of the proof is devoted to
showing~\eqref{equation:comparison}. 
Let $x\in\mc{N}(S,\eps)$ with $\eps$ to be determined. Let $x_\star$
be the minimum closest to $x$, i.e. $\ltwo{x-x_\star}=\dist(x,S)$.  
Recall that the penalized objective is
\begin{equation*}
  \vphi(x) = \vphi_f(x) + \lambda\vphi_g(x),
\end{equation*}
where
\begin{equation*}
  \vphi_f(x) = f(c(x))~~{\rm and}~~\vphi_g(x) =
  \ltwo{Ac(x)-b}.
\end{equation*}
As the sub-differential of $\vphi_g$ depends on whether
$Ac(x)=b$, we show regularity for these two cases
separately. 

\paragraph{Case 1: $x$ is feasible} In this case, we have
$Ac(x)=b$, so
\begin{equation*}
  \partial\vphi_g(x) = \set{\grad c(x) A^\top
    z~\big\vert~\ltwo{z}\le 1}.
\end{equation*}
The sub-differential of $\vphi$ is thus
\begin{equation*}
  \partial\vphi(x) = \set{ \grad\vphi_f(x) +
      \grad c(x) A^\top(\lambda z)~\big\vert~\ltwo{z} \le 1}.
\end{equation*}
As we increase $\lambda$, the sub-differential set will get larger,
but the minimum norm sub-differential will stay constant after
$\lambda$ passes a threshold. This limiting sub-differential is
the projection of $\grad c(x)\grad f(c(x))$ onto the orthogonal
complement of the row space of $A\grad c(x)^\top$. Indeed,
let $G(x)$ be the minimum norm element in $\partial\vphi(x)$. We
claim that for $\lambda$ sufficiently large, we have the projection
relation
\begin{equation}\label{equation:projected-gradient}
  G(x) = \left(I_{n\times n} - \grad c(x) A^\top [\grad
    c(x)A^\top]^\dagger \right) \grad\vphi_f(x).
\end{equation}
Indeed, first-order optimality condition of $x_\star$ implies
that there exists $y_\star\in\R^k$ such that
\begin{equation*}
  \grad\vphi_f(x_\star) + \grad c(x_\star) A^\top y_\star = 0.
\end{equation*}
Thus $G(x_\star)=0$. If we define the optimal coefficient vector
\begin{equation*}
  z(x) \defeq \argmin_{z\in\R^m} \lfro{\grad\vphi_f(x) +
    \grad c(x) A^\top z} = \grad c(x) A^\top [\grad
    c(x)A^\top]^\dagger \grad\vphi_f(x),
\end{equation*}
then $z(x_\star)=y_\star$, and $\lambda\ge\ltwo{y_\star}$ suffices to
guarantee that the minimum
norm element in $\partial\vphi(x_\star)$ is $G(x_\star)=0$. By
Assumption~\ref{assumption:constraint-non-degeneracy}, we have
$\sigma_{\min}(A\grad c(x)^\top)=\sigmacq>0$ is bounded
away from zero for $x\in\mc{N}(S,\eps_0)$. Combined with the
smoothness of $f$, we see that $x\mapsto z(x)$ is differentiable for
$x\in\mc{N}(S,\eps_0)$~\cite{StewartSu90}, so by compactness of $S$,
there exists $\eps$ such that
\begin{equation*}
  \sup_{x\in\mc{N}(S,\eps)} \ltwo{z(x)} \le
  2\sup_{x\in S}\ltwo{z(x)} \defeq \Lambdaproj < \infty.
\end{equation*}
Consequently, taking
$\lambda\ge\Lambdaproj$, we have $\lambda\ge\ltwo{z(x)}$ and the
projection relation~\eqref{equation:projected-gradient} for all
$x\in\mc{N}(S,\eps)$.

Let us further observe that if we view $\vphi_f|_{\mc{M}}$ as a smooth
function on the Riemannian manifold 
$\mc{M}=\{x\in\R^n:Ac(x)=b\}$, and let $\rgrad
\vphi_f(x)$ be the vector representation of the Riemannian gradient
of $\vphi_f(x)$ on $\mc{M}$, then (recalling the
definition~\eqref{equation:riemannian-grad-hess})
\begin{equation*}
  G(x) = \proj_{\mc{T}_x\mc{M}}(\grad \vphi_f(x)) = \rgrad
  \vphi_f(x).
\end{equation*}
This is because the tangent space $\mc{T}_x\mc{M}$ has the
representation
\begin{equation*}
  \mc{T}_x\mc{M} = \set{w\in\R^n:  A\grad c(x)^\top w=0},
\end{equation*}
which is the orthogonal complement of $\{\grad c(x)A^\top
z:z\in\R^k\}$.

With this relation in hand, we now show that $G(x)$ satisfies
\emph{strong star-convexity}
\begin{equation}
  \label{equation:riemannian-star-convexity}
  \<\rgrad \vphi_f(x), x-x_\star\> \ge c\cdot \ltwo{x - x_\star}^2.
\end{equation}
We do this by performing a (Euclidean version of) Riemannian Taylor
expansion on $\mc{M}$, stated in the following lemma. We believe this
result is standard; for completeness we give a
proof in Appendix~\ref{appendix:proof-riemannian-taylor-expansion}.
\begin{lemma}
  \label{lemma:riemannian-taylor-expansion}
  Let $f:\R^n\to\R$ and $F:\R^n\to\R^m$ be smooth functions with
  $x_\star$ a local minimizer of the constrained problem
  \begin{equation*}
    \begin{aligned}
      \minimize & ~~ f(x) \\
      \subjectto & ~~ F(x) = 0.
    \end{aligned}
  \end{equation*}
  Assume $\sigma_{\min}(\grad F(x_\star))\ge \gamma>0$, and consider
  the analytical formulae for Riemannian gradient and Hessian
  \begin{align*}
    \rgrad f(x) = & \grad f(x) - \sum_{i=1}^m \lambda_i(x) \grad F_i(x),\\
    \rhess f(x) = & \grad^2 f(x) - \sum_{i=1}^m \lambda_i(x)
                         \grad^2 F_i(x),
  \end{align*}
  where $\lambda_i(x)=[\grad F(x)^\dagger \grad f(x)]_i$. There exists
  $\delta>0$
  such that for all $x\in\ball(x_\star,\delta)\cap\set{x:F(x)=0}$, the
  following holds uniformly:
  \begin{align}
    f(x) - f(x_\star)
    = & 1/2\cdot\langle \rhess f(x_\star), (x - x_\star)^{\otimes
        2}\rangle + O(\dl x - x_\star
        \dl_2^3), \label{eqn:f_r_expansion} \\ 
    \langle \rgrad f(x), x - x_\star\rangle
    = & \langle \rhess f(x_\star), (x - x_\star)^{\otimes 2}\rangle
        + O(\dl x - x_\star \dl_2^3). 
  \end{align}
  The neighborhood size $\delta$ and the leading constant in the
  big-$O$ only depend on $(\lipf,\smoothf,\smoothF,\gamma)$.
\end{lemma}
We now apply Lemma~\ref{lemma:riemannian-taylor-expansion} with
$\vphi_f$ and $Ac(\cdot)-b$ simultaneously for all $x_\star\in S$.
Noticing that $(\lipf, \smoothf, \smoothF, \sigmacq)$ are assumed to be
uniformly bounded in $\mc{N}(S,\eps_0)$, there exists $\eps>0$ and
$\rho<\infty$ such that for all $x\in\mc{N}(S,\eps)$ and aligned
minimum $x_\star$, we have
\begin{equation*}
  \<\rgrad \vphi_f(x), x-x_\star\> \ge 2(\vphi_f(x) -
  \vphi_f(x_\star)) - \rho\ltwo{x-x_\star}^3.
\end{equation*}
As the leading term $\vphi_f(x)-\vphi_f(x_\star)$ grows quadratically with
$\ltwo{x-x_\star}$, we have
\begin{equation*}
  \<\rgrad\vphi_f(x), x-x_\star\> \ge \vphi_f(x) -
  \vphi_f(x_\star) = \vphi(x) - \vphi(x_\star)
\end{equation*}
for sufficiently small $\eps$. This gives
\begin{equation*}
  \vphi(x) - \vphi_\star \le \<\rgrad\vphi_f(x), x-x_\star\> \le
  \ltwo{\rgrad\vphi_f(x)} \cdot \ltwo{x-x_\star} = \dist(0,
  \partial\vphi(x)) \cdot \dist(x, S),
\end{equation*}
which verifies~\eqref{equation:comparison} with $\ell=1$.

\paragraph{Case 2: $x$ is infeasible} In this case, the penalty
$\vphi_g(x)$ is differentiable at $x$ and
\begin{equation*}
  \grad\vphi_g(R) = \grad c(x) A^\top z,~~z =
  \frac{Ac(x)-b}{\ltwo{Ac(x)-b}}.
\end{equation*}
By Assumption~\ref{assumption:constraint-non-degeneracy}, for
$x\in\mc{N}(S,\eps_0)$ we have $\sigma_{\min}(A\grad c(x)^\top)\ge
\sigmacq$, implying that
\begin{equation*}
  \ltwo{\grad\vphi_g(x)} = \ltwo{\grad c(x) A^\top z} \ge
  \sigmacq\cdot\ltwo{z} = \sigmacq.
\end{equation*}
As we have
$\ltwo{\grad\vphi_f(x)}\le \lipphif$ for all $x\in\mc{N}(S,\eps_0)$
and thus
\begin{equation*}
  \ltwo{\grad\vphi(x)} \ge \ltwo{\grad\vphi_f(x) +
    \lambda\grad\vphi_g(x)} \ge \lambda\ltwo{\grad\vphi_g(x)} - \lipphif.
\end{equation*}
When $\lambda\ge\Lambdaqg=2\lipphif/\sigmacq$, we get
\begin{equation*}
  \lipphif \le \frac{\sigmacq\lambda}{2} \le
  \frac{\lambda}{2}\ltwo{\grad\vphi_g(x)},
\end{equation*}
which implies that the norm
$\ltwo{\grad\vphi(x)}$ is lower bounded by
$\lambda\ltwo{\grad\vphi_g(x)}/2$ and hence
\begin{equation}
  \label{equation:infeasible-bound-1}
  \ltwo{\grad\vphi(x)} \cdot \dist(x,S) \ge
  \frac{\lambda}{2} \ltwo{\grad\vphi_g(x)} \cdot \dist(x,S).
\end{equation}
Now, applying Lemma~\ref{lemma:norm-bound}(2) with $\delta=1/4$, in a
neighborhood $\mc{N}(S,\eps)$, we can lower bound the above as
\begin{equation*}
  \frac{\lambda}{2} \ltwo{\grad\vphi_g(x)} \cdot \dist(x,S) \ge
  \frac{(1-\delta)\lambda}{2} \ltwo{Ac(x)-b} =
  \frac{\lambda}{4}\ltwo{Ac(x)-b}.
\end{equation*}
Thus we can then upper bound the objective growth as
\begin{align*}
  & \quad \vphi(x) - \vphi_\star = (\vphi_f(x) - \vphi_f(x_\star)) +
    \lambda\ltwo{Ac(x)-b} \\
  & \le \lipphif\ltwo{x-x_\star} +
    4\ltwo{\grad\vphi(x)} \cdot \ltwo{x-x_\star} \\
  & \le 5\ltwo{\grad\vphi(x)} \cdot \ltwo{x-x_\star},
\end{align*}
which verifies~\eqref{equation:comparison} with $\ell=5$.

\paragraph{Putting together} We conclude
that~\eqref{equation:comparison} holds in a neighborhood of
$S$ whenever
\begin{align*}
  \lambda
  & \ge \max\set{\Lambdaqg, \Lambdaproj}
\end{align*}
with constant $\ell = \max\set{1, 5} = 5$.

\subsection{Proof of Lemma~\ref{lemma:riemannian-taylor-expansion}}
\label{appendix:proof-riemannian-taylor-expansion}

\paragraph{Step 1}

First, note that $F(x) = 0$ and $F(x_\star) = 0$, expanding $F_i(x)$ at $x_\star$, we have
\begin{align}
F_i(x) = F_i(x_\star) + \langle \grad F_i(x_\star), x - x_\star\rangle + 1/2 \cdot \langle \grad^2 F_i(x_\star), (x - x_\star)^{\otimes 2} \rangle + O(\dl x - x_\star \dl_2^3).
\end{align}
This gives
\begin{align}
\langle \grad F_i(x_\star), x - x_\star\rangle + 1/2 \cdot \langle \grad^2 F_i(x_\star), (x - x_\star)^{\otimes 2} \rangle =& O(\dl x - x_\star \dl_2^3). \label{eqn:F_expansion_3rd}
\end{align}

Second, we expand $f(x)$ at $x_\star$, which gives
\begin{align}\label{eqn:f_expansion}
f(x) = f(x_\star) + \langle \grad f(x_\star), x - x_\star\rangle + 1/2 \cdot \langle \grad^2 f(x_\star), (x - x_\star)^{\otimes 2}\rangle + O(\dl x - x_\star \dl_2^3). 
\end{align}

Combining equation (\ref{eqn:F_expansion_3rd}) and (\ref{eqn:f_expansion}), we have
\begin{equation}
\begin{aligned}
f(x) =& f(x_\star) + \langle \grad f(x_\star) - \sum_{i=1}^m \lambda_i(x_\star) \grad F_i(x_\star), x - x_\star \rangle \\
&+ 1/2 \cdot \langle \grad^2 f(x_\star) - \sum_{i=1}^m \lambda_i(x_\star) \grad^2 F_i(x_\star), (x - x_\star)^{\otimes 2} \rangle + O(\dl x - x_\star \dl_2^3)\\
=& f(x_\star) + \langle \rgrad f(x_\star), x - x_\star\rangle +1/2 \cdot \langle \rhess f(x_\star), (x - x_\star)^{\otimes 2}\rangle + O(\dl x - x_\star \dl_2^3). 
\end{aligned}
\end{equation}
Since $x_\star$ is a local minimizer, we have $\rgrad f(x_\star) = 0$. Therefore, we proved equation (\ref{eqn:f_r_expansion}).

\paragraph{Step 2}

First, since $f, F$ is $C^3$ near $x_\star$ and the singular values of $\grad F(x)$ are lower bounded, $\lambda_i(x)$ is local Lipschitz around $x_\star$, and we have
\begin{align}
\lambda_i(x) - \lambda_i(x_\star)  = O(\dl x - x_\star \dl_2).
\end{align}
Expanding $F_i(x_\star)$ around $x$ gives
\begin{align}
F_i(x_\star) = F_i(x) + \langle \grad F_i(x), x_\star - x\rangle + O(\dl x - x_\star \dl_2^2).
\end{align}
Note $F_i(x) = F_i(x_\star) = 0$, this gives
\begin{align}
\langle \grad F_i(x), x - x_\star\rangle = O( \dl x - x_\star \dl_2^2). 
\end{align}

Therefore, we have
\[
\begin{aligned}
\langle \rgrad f(x), x - x_\star\rangle=& \langle \grad f(x) - \sum_{i=1}^m \lambda_i(x_\star) \grad F_i(x), x - x_\star\rangle + \sum_{i=1}^m [\lambda_i(x) - \lambda_i(x_\star)] \langle \grad F_i(x), x - x_\star\rangle\\
=& \langle \grad f(x) - \sum_{i=1}^m \lambda_i(x_\star) \grad F_i(x), x - x_\star\rangle + O(\dl x - x_\star \dl_2^3)\\
=& \langle \grad f(x_\star) + \grad^2 f(x_\star)[x - x_\star]  \\
&- \sum_{i=1}^m \lambda_i(x_\star) \{\grad F_i(x_\star) + \grad^2 F_i (x_\star)[x - x_\star] \}+ O(\dl x - x_\star \dl_2^2), x - x_\star \rangle + O(\dl x - x_\star \dl_2^3)\\
=& \langle \rgrad f(x_\star) + \rhess f(x_\star)[x - x_\star], x - x_\star\rangle + O(\dl x - x_\star \dl_2^3)\\
=& \langle \rhess f(x_\star), (x - x_\star)^{\otimes 2}\rangle + O(\dl x - x_\star \dl_2^3). 
\end{aligned}
\]

\paragraph{Step 3} Throughout our analysis, the big-$O$ terms in our
Taylor expansions only depend on $(\lipf, \smoothf, \smoothF,
\gamma)$. Hence the neighborhood size $\delta$ and the leading
constants in the big-$O$'s will also only depend
on these parameters (and no other properties of the particular point
$x_\star$).
\section{Proofs for Section~\ref{section:algorithm}}
\subsection{Preliminaries on composite optimization}
\label{appendix:review-composite-opt}
For any $x\in\R^n$ and $t\le\frac{1}{L\beta}$, let
\begin{equation}
  \label{equation:prox-linear-iterate}
  x^t = \argmin_{\wt{x}\in\R^n}\vphi_t(x;\wt{x}) =
  \argmin_{\wt{x}\in\R^n}\set{\vphi(x;\wt{x}) +
    \frac{1}{2t}\ltwo{\wt{x}-x}^2}
\end{equation}
be the next iterate of the prox-linear algorithm. Define the
prox-linear gradient mapping $\gm_t:\R^n\to\R^n$ as
\begin{equation*}
  \gm_t(x) \defeq \frac{1}{\alpha}(x^t - x).
\end{equation*}
One can easily verify that $\gm_t(x)=0$ is equivalent to that $x$
is stationary to problem~\eqref{problem:penalized-bm}. In general,
$\gm_t(x)$ it is the direction in which the prox-linear
algorithm moves.

To analyze the convergence rate, we introduce the error bound
condition.
\begin{definition}[Error bound condition]
  We say that $\gm_t(x)$ satisfies the error bound condition
  around a point $x_\star$ with parameter $\gamma>0$, if there exists
  $\eps>0$ such that
  \begin{equation*}
    \dist(x, S) \le \gamma\ltwo{\gm_t(x)}
  \end{equation*}
  holds for all $x\in\ball(x_\star,\eps)$.
\end{definition}

The following Lemmas, established in~\citet{DrusvyatskiyLe16}, will be
useful in our convergence proof.
\begin{lemma}
  \label{lemma:regularity-implies-error-bound}
  If a convex composite $\vphi$ is sub-differentially regular at
  $x_\star$ with constant $\ell$, then $\gm_t$ satisfies the
  error bound condition at $x_\star$ with constant
  $\gamma=(3L\beta t+2)\ell+2t$.
\end{lemma}
\begin{lemma}
  \label{lemma:near-stationarity}
  Suppose $\vphi(\cdot)=h(c(\cdot))$, where $h$ is $L$-Lipschitz
  continuous and $c$ is $\beta$-smooth. Then for any $t>0$, there
  exists a point $\what{x}$ satisfying the properties
  \begin{enumerate}[(i)]
  \item (point proximity) $\ltwo{x^t-\what{x}} \le \ltwo{x^t-x}$.
  \item (value proximity) $\vphi(\what{x}) - \vphi(x^t) \le
    \frac{t}{2}(L\beta t+1)\ltwo{\gm_t(x)}^2$.
  \item (near-stationarity) $\dist(0, \partial\vphi(\what{x})) \le
    (3L\beta t+2)\ltwo{\gm_t(x)}$.
  \end{enumerate}
\end{lemma}
\begin{lemma}[Descent bound]
  \label{lemma:descent-bound}
  Taking $t=(L\beta)^{-1}$, the prox-linear iterate $x\mapsto x^t$
  satisfies
  \begin{equation}
    \label{equation:descent-bound}
    \vphi(x^t) \le \vphi(x) - \frac{1}{2L\beta}\lfro{\gm_t(x)}^2.
  \end{equation}
\end{lemma}

\subsection{Proof of Theorem~\ref{theorem:linear-convergence}}
\label{appendix:proof-prox-linear}

We first formally define the proximity property assumed in
Theorem~\ref{theorem:linear-convergence}.
\begin{definition}[Proximity property]
  \label{definition:proximity}
  The prox-linear algorithm is said to satisfy the proximity property
  on the composite objective $\vphi$ if there exists a function
  $d(\eps)>0$ and some $\eps_0>0$ such that for all $\eps\le\eps_0$,
  initializing in a $d(\eps)$-neighborhood of $S$, the prox-linear
  iterates with $t=(L\beta)^{-1}$ never leave the $\eps$-neighborhood.
\end{definition}
The proximity property guarantees that the prox-linear method stays close to
the local minimizing set $S$ once initialized close to it, and is
typically required for showing the local convergence as it cannot be
otherwise deduced from generic regularity
conditions~\cite{DrusvyatskiyLe16}. We note, however, that it can
typically be verified on problems where the smooth map $c$ has
additional structures, such as the matrix problem considered in
Section~\ref{section:matrix-growth} (see
Appendix~\ref{appendix:proof-corollary-rate} for such an argument.)

\begin{proof-of-theorem}[\ref{theorem:linear-convergence}]
Let $\dist(x^0,S)\le d(\eps)$, then by assumption we have
$\dist(x^k,S)\le\eps$ for all $k$. We now analyze one
iterate. Consider a point $x\in\mc{N}(S,\eps)$ and let
$x^t\in\mc{N}(S,\eps)$ denote the next iterate, where
$t=(L\beta)^{-1}$. By Lemma~\ref{lemma:near-stationarity}, there
exists some $\what{x}$ such that
\begin{equation}
  \label{equation:xhat-properties}
  \ltwo{x^t - \what{x}} \le \ltwo{x^t - x}~~{\rm
    and}~~\dist(0,\partial\vphi(\what{x})) \le (3L\beta
  t+2)\ltwo{\gm_t(x)} \le 5\ltwo{\gm_t(x)}.
\end{equation}
In particular, the first bound implies that
\begin{equation*}
  \dist(\what{x},S) \le \ltwo{\what{x}-x^t} +
  \dist(x^t,S) \le \ltwo{x-x^t} + \dist(x^t,S)
  \le 3\eps.
\end{equation*}
Choosing $\eps\le\eps_0/3$, as long as $\dist(x,S)\le\eps$, we will
have $\dist(\what{x},S)\le\eps_0$, so all the geometric properties
(norm convexity, sub-differential regularity, and the error bound
condition) will hold for both $\what{x}$ and $x$ (consider all
$x_\star\in S$ and use the finite sub-cover property of $S$). Building
on these, we can upper bound the optimality gap at $\what{x}$ as
\begin{equation}
  \label{equation:vphi-upper-bound}
  \begin{aligned}
    \vphi(\what{x}) - \vphi_\star
    & \stackrel{(i)}{\le} \ell \cdot
    \dist(0,\partial\vphi(\what{x})) \cdot \dist(\what{x},S) \\
    & \stackrel{(ii)}{\le} 5\ell\ltwo{\gm_t(x)} \cdot \dist(\what{x},S) \\
    & \stackrel{(iii)}{\le} 5\ell\ltwo{\gm_t(x)} \cdot
    (\ltwo{\what{x}-x} + \dist(x,S)) \\
    & \stackrel{(iv)}{\le} 5\ell\ltwo{\gm_t(x)} \cdot (2\ltwo{x^t-x} +
    \dist(x,S)) \\
    & = 5\ell\ltwo{\gm_t(x)} \cdot (2t\ltwo{\gm_t(x)} + \dist(x,S)).
  \end{aligned}
\end{equation}
In the above, (i) is norm convexity at $\what{x}$, (ii) uses the
near-stationarity condition~\eqref{equation:xhat-properties}, (iii) is
triangle inequality, and (iv) is another triangle inequality plus the
distance bound~\eqref{equation:xhat-properties}.
Now, by Theorem~\ref{theorem:quadratic-implies-regularity},
sub-differential regularity holds at $x$ with constant
$\wt{\ell}=2\ell/\alpha_\vphi$. Applying
Lemma~\ref{lemma:regularity-implies-error-bound}, the error
bound condition also holds at $x$ with constant
\begin{equation*}
  \gamma = (3L\beta t+2)\wt{\ell} + 2t \le 5\wt{\ell} + 2t,
\end{equation*}
which gives that
\begin{equation*}
  \dist(x,S) \le \gamma\cdot \ltwo{\gm_t(x)} \le
  \left(5\wt{\ell}+2t\right) \ltwo{\gm_t(x)}.
\end{equation*}
Substituting this into the preceding upper
bound~\eqref{equation:vphi-upper-bound} gives
\begin{equation}
  \label{equation:optgap-upper-bound}
  \vphi(\what{x}) - \vphi_\star \le 5\ell\ltwo{\gm_t(x)}^2 \left(
    2t + 5\wt{\ell} +2t \right) \le 5\ell\left( 5\wt{\ell} +
    \frac{4}{L\beta} \right)\ltwo{\gm_t(x)}^2.
\end{equation}
On the other hand, similar to ~\citep[proof of Theorem
6.3]{DrusvyatskiyLe16}, we can lower bound the optimality gap as
\begin{equation}
  \label{equation:optgap-lower-bound}
  \begin{aligned}
    \vphi(\what{x}) - \vphi_\star
    & \stackrel{(i)}{\ge} \vphi(x;\what{x}) -
    \frac{L\beta}{2}\ltwo{x-\what{x}}^2 - \vphi_\star \\
    & = \left(\vphi(x;\what{x}) +
      \frac{1}{2t}\ltwo{\what{x}-x}^2\right) -
    \frac{L\beta+t^{-1}}{2} \ltwo{x-\what{x}}^2 - \vphi_\star \\
    & \stackrel{(ii)}{\ge} \vphi(x^t) - \frac{L\beta+t^{-1}}{2}
    \ltwo{x-\what{x}}^2 - \vphi_\star \\
    & \stackrel{(iii)}{\ge} \vphi(x^t) - \vphi_\star - 2(L\beta+t^{-1})
    \ltwo{x^t-x}^2 \\
    & = \vphi(x^t) - \vphi_\star - \frac{4}{L\beta} \cdot
    \ltwo{\gm_t(x)}^2.
  \end{aligned}
\end{equation}
In the above, (i) uses the quadratic approximation
property~\eqref{equation:quadratic-approximation}, (ii) uses the fact
that $x^t$ minimizes $y\mapsto \vphi(x;y)+\ltwo{y-x}^2/(2t)$, and
(iii) uses the triangle inequality
$\ltwo{\what{x}-x}\le\ltwo{\what{x}-x^t}+\ltwo{x^t-x}\le
2\ltwo{x^t-x}$.
Combining~\eqref{equation:optgap-upper-bound}
and~\eqref{equation:optgap-lower-bound}, we get
\begin{align*}
  \vphi(x^t) - \vphi_\star
  & \le \vphi(\what{x}) - \vphi_\star +
    \frac{4}{L\beta}\ltwo{\gm_t(x)}^2 \\
  & \le 5\ell\left(5\wt{\ell} + \frac{4}{L\beta}\right)
    \ltwo{\gm_t(x)}^2 + \frac{4}{L\beta}\ltwo{\gm_t(x)}^2 \\
  & = \frac{\ltwo{\gm_t(x)}^2}{L\beta} \cdot \left( 4 + 20\ell +
    25\ell\wt{\ell}L\beta \right) \\
  & \le \frac{\ltwo{\gm_t(x)^2}}{L\beta} \cdot \left(4 + 20\ell +
    50\ell^2 \frac{L\beta}{\alpha_\vphi}\right).
\end{align*}
Finally, appying the descent bound~\eqref{equation:descent-bound},
that is,
\begin{equation*}
  \vphi(x^t) - \vphi(x) \le -\frac{1}{2L\beta}\ltwo{\gm_t(x)}^2,
\end{equation*}
and performing standard algebraic manipulations, we get
\begin{align*}
  \vphi(x^t) - \vphi_\star
  \le \left( 1 -
  \frac{1}{9+40\ell+100\ell^2L\beta/\alpha_\vphi}\right) (\vphi(x)
  - \vphi_\star).
\end{align*}
This is the desired result.
\end{proof-of-theorem}

\subsection{Relationship between norm convexity, sub-differential
  regularity, and tilt stability}
\label{appendix:relationship-tilt-stability}
We shall compare norm convexity with tilt stability, a variational
condition that also guarantees linear convergence of the prox-linear
algorithm~\cite{DrusvyatskiyLe16}.

A local minimum $x_\star$ is said
to be $\ell$-tilt-stable if there exists a neighborhood $\mc{X}$ of
$x_\star$ such that for all small enough $v$, the solution mapping
\begin{equation*}
  v \mapsto \argmin_{x\in\mc{X}} \vphi(x) - \<v, x\>
\end{equation*}
is single-valued and $\ell$-Lipschitz. Note that this requires the
solution to be unique, which won't hold in our case. Norm convexity is
similar to tilt stability but a bit more relaxed -- it allows the
local minimum to be non-unique but still guarantees that function
growth near the minimizing set can be upper bounded by the gradient
times the distance to the minimizing set. A concrete example is the
convex function $x\mapsto [|x|-1]_+$, whose minimizing set is
$[-1,1]$, around which the function is not tilt-stable but norm-convex
with constant $1$.

\section{Proofs for Section~\ref{section:matrix-growth}}
\subsection{Background on matrix factorization and Procrustes
  distance}
\label{appendix:background-matrix}
We provide some background on the Procrustes problem, based
on~\citet{TuBoSiSoRe16}.
Let $R \in \R^{n
  \times r}$ and $Q \in \R^{n \times r}$ be arbitrary matrices, $n \ge
r$. The Procrustes problem is the alignment problem
\begin{equation*}
  \minimize_{\Omega \in \R^{r \times r}, \Omega^T \Omega = I_r}
  \lfro{R - Q\Omega}^2.
\end{equation*}
By expanding the Frobenius norm, this is equivalent to maximizing $\<R,
Q\Omega\> = \tr(R^TQ \Omega)$ over orthogonal matrices
$\Omega$. Writing the singular value decomposition of $R^T Q = U
\Sigma V^T$, where $U, V \in \R^{r\times r}$ are orthogonal and
$\Sigma \in \diag(\R_+^r)$, we see
\begin{equation*}
  \tr(R^T Q \Omega)
  = \tr(\Sigma V^T \Omega U)
  \le \tr(\Sigma),
\end{equation*}
where we have used von-Neumann's trace inequality, with equality
achieved by $\Omega = V U^T$.
If we define the difference $\Delta = R - Q\Omega$, then we have
\begin{equation*}
  R^T Q \Omega
  = U \Sigma V^T V U^T = U \Sigma U^T
  = U V^T V \Sigma U^T = \Omega^T Q^T R \succeq 0
\end{equation*}
and that
$\Delta^T Q \Omega = (Q \Omega)^T \Delta$ is symmetric. As a direct
consequence, we have
\begin{equation*}
  \dist(R,S) = \min_{\Omega\in\R^{r\times
      r},\Omega^\top\Omega=I_{r\times r}} \lfro{R-R_\star\Omega}^2
\end{equation*}
equals the optimal value of the Procrustes problem.
We say that $R$ and $R_\star\in S$ are optimally aligned if
$\lfro{R-R_\star}=\dist(R,S)$, i.e. when the minimizing $\Omega$ is
the identity.

\subsection{Proof of Lemma~\ref{lemma:global-optimality}}
\label{appendix:proof-global-optimality}
Let $g(X)=\ltwo{\mc{A}(X)-b}$ and define
\begin{equation*}
  h(X) = f(X) + \lambda\ltwo{\mc{A}(X)-b}.
\end{equation*}
then $\vphi_f(X)=f(RR^\top)$, $\vphi_g(R)=g(RR^\top)$, and
$\vphi(R)=h(RR^\top)$. By the KKT condition, we have that $X_\star$ is
the unique minimum of
\begin{equation*}
  L(X,y_\star) = f(X) + \<y_\star, \mc{A}(X)-b\>.
\end{equation*}
This implies that for all $X\succeq 0$, $X\neq X_\star$, taking
$\lambda\ge\ltwo{y_\star}$, 
\begin{align*}
  & \quad h(X) = f(X) + \lambda\ltwo{\mc{A}(X)-b} > f(X) +
    \ltwo{y_\star}\ltwo{\mc{A}(X)-b} \\
  & > f(X) + \<y_\star, \mc{A}(X)-b\> = L(X, y_\star)
    \ge L(X_\star, y_\star) = f(X_\star) = h(X_\star).
\end{align*}
As $\vphi(R)=h(RR^\top)$, the minimizing set of $\vphi$ is
$\set{R_\star:R_\star R_\star^\top=X_\star}$. This completes the proof.

\subsection{Proof of Lemma~\ref{lemma:factorized-quadratic-growth}}
\label{appendix:proof-factorize-growth}
Take any $R$ such that $\mc{A}(RR^\top)=b$. Plugging $X=RR^\top$
into the quadratic growth condition, we get that
\begin{equation*}
  f(RR^\top) - f(R_\star R_\star^\top) \ge
  \frac{\alpha_f}{2}\lfro{RR^\top - R_\star R_\star^\top}^2.
\end{equation*}
Recall the characterization of the distance to $S_\star$:
\begin{equation*}
  \dist(R,S_\star) = \dist(R,\set{R_\star Q:Q\in O(r)}) = \min_{Q\in
    O(r)} \lfro{R - R_\star Q}.
\end{equation*}
Applying Lemma~\ref{lemma:matrix-factorization-bound}, we have
\begin{equation*}
  \lfro{RR^\top - R_\star R_\star^\top}^2 \ge
  2(\sqrt{2}-1)\sigma_r^2(R_\star)\dist(R, S_\star)^2,
\end{equation*}
which leads to the conclusion.

\subsection{Proof of
  Lemma~\ref{lemma:overspecified-factorization-bound}}
\label{appendix:proof-overspecified-factorization}
We will show that
\begin{equation*}
  \lfro{UU^\top - U_\star U_\star^\top}^2
  \ge
  2(\sqrt{2}-1)\max\set{\lfro{U_\star\Delta^\top},
    \frac{1}{3}\lfro{\Delta\Delta^\top}}^2
  \ge \frac{2(\sqrt{2}-1)}{9r}\lfro{\Delta}^4.
\end{equation*}
The first part of the max is already shown
in Lemma~\ref{lemma:matrix-factorization-bound}. Now, if
$\lfro{U_\star\Delta^\top} \le \frac{1}{3}\lfro{\Delta\Delta^\top}$,
we have
\begin{equation*}
  \lfro{UU^\top - U_\star U_\star^\top} = \lfro{U_\star\Delta^\top
    + \Delta U_\star^\top + \Delta\Delta^\top} \ge
  \lfro{\Delta\Delta^\top} - 2\lfro{U_\star\Delta^\top} \ge
  \frac{1}{3}\lfro{\Delta\Delta^\top}.
\end{equation*}
Noting that $2(\sqrt{2}-1)<1$, we get the second part of the
max. The second inequality follows from the second max and that
\begin{equation*}
  \lfro{\Delta\Delta^\top}^2 = \sum_{i=1}^r \sigma_i(\Delta)^4 \ge
  \frac{1}{r}\left(\sum_{i=1}^r\sigma_i(\Delta)^2\right)^2 =
  \frac{1}{r}\lfro{\Delta}^4.
\end{equation*}

\subsection{Proof of Lemma~\ref{lemma:overspecified-growth}}
\label{appendix:proof-linear-overspecified}
The goal is to show that for $U\in\mc{M}_r$ near $U_\star$, the function
\begin{equation*}
  \vphi(U) \defeq \lfro{Q_2^\top U}^2 = \lnuc{Q_2^\top UU^\top Q_2} =
  \lnuc{Q_2^\top XQ_2}
\end{equation*}
grows quadratically in $\dist(U,U_\star)^2$. This is a quadratic
function on the feasible manifold
\begin{equation*}
  \mc{M}_r \defeq \set{U\in\R^{n\times r}:\mc{A}(UU^\top)=b}.
\end{equation*}
As $\vphi(U)$ depends on $U$ only through $UU^\top$, WLOG we can
assume that $U$ is aligned with $U_\star$, so that
$\dist(U,U_\star)=\lfro{U-U_\star}$. Let 
\begin{equation*}
  U_\star = Q_1V_\star, ~~~ \Delta \defeq U-U_\star = Q_1V_\Delta +
  Q_2W_\Delta
\end{equation*}
for some $V_\star,V_\Delta\in\R^{r_\star\times r}$ and
$W_\Delta\in\R^{(n-r_\star)\times r}$. As $U$ and $U_\star$ are
aligned, by properties of the Procrustes problem, we require
\begin{equation}
  \label{equation:delta-ustar-aligned}
  \Delta^\top U_\star = V_\Delta^\top Q_1^\top Q_1V_\star +
  W_\Delta^\top Q_2^\top Q_1V_\star = V_\Delta^\top V_\star \succeq
  0.
\end{equation}
We have
\begin{equation*}
  \vphi(U) = \vphi(U_\star+\Delta) =
  \lfro{Q_2^\top(Q_1V_\Delta+Q_2W_\Delta)}^2 = 
  \lfro{W_\Delta}^2
\end{equation*}
and
\begin{equation*}
  \lfro{\Delta}^2 = \lfro{Q_1V_\Delta + Q_2W_\Delta}^2 =
  \lfro{V_\Delta}^2 + \lfro{W_\Delta}^2.
\end{equation*}
Hence, to show quadratic growth, it suffices to show that
$\lfro{V_\Delta}^2$ is upper bounded by constant times
$\lfro{W_\Delta}^2$. To do this, we use the property of the tangent
space $\mc{T}_{U_\star}\mc{M}_r$. For $U$ near $U_\star$, the
difference matrix $\Delta=U-U_\star$ approaches
$\mc{T}_{U_\star}\mc{M}_r$. Indeed, we have the following result,
whose proof can be found in
Section~\ref{section:proof-quadratic-normal-component}.
\begin{lemma}
  \label{lemma:quadratic-normal-component}
  There exists a constant $C>0$ such that
  \begin{equation}
    \label{equation:normal-component-bound}
    \sum_{i=1}^k \<A_iU_\star, \Delta\>^2 \le
    C\lfro{\Delta}^4~~\textrm{for all}~U\in
    \mc{M}_r.
  \end{equation}
\end{lemma}

Observe that
\begin{equation*}
  \<A_iU_\star, \Delta\> = \<A_iQ_1V_\star, Q_1V_\Delta+Q_2W_\Delta\>
  = \<Q_1^\top A_iQ_1, V_\Delta V_\star^\top\> + \<Q_2^\top A_iQ_1,
  W_\Delta V_\star^\top\>.
\end{equation*}
Rearranging, we get
\begin{equation*}
  \<Q_1^\top A_iQ_1, V_\Delta V_\star^\top\> = -\<Q_2^\top A_iQ_1,
  W_\Delta V_\star^\top\> + \<A_iU_\star, \Delta\>.
\end{equation*}
Squaring and summing over $i$ and using
Lemma~\ref{lemma:quadratic-normal-component}, we get
\begin{equation}
  \label{equation:dn-upper-bound}
  \begin{aligned}
    & \quad \sum_{i=1}^k \<Q_1^\top A_i Q_1, V_\Delta V_\star^\top\>^2
    \le 2\sum_{i=1}^k \<Q_2^\top A_iQ_1, W_\Delta V_\star^\top\>^2 +
    2\sum_{i=1}^k \<A_iU_\star, \Delta\>^2 \\
    & \le 2\sum_{i=1}^k \lfro{Q_2^\top A_iQ_1}^2 \cdot \lfro{W_\Delta
      V_\star^\top}^2 + 2C\lfro{\Delta}^4 \\
    & \le 2\left(\sum_{i=1}^k \lfro{Q_2^\top A_iQ_1}^2\right) \cdot
    \lambda_{\max}\left(V_\star V_\star^\top\right) \lfro{W_\Delta}^2
    + 2C(\lfro{V_\Delta}^2 + \lfro{W_\Delta}^2)^2.
  \end{aligned}
\end{equation}
On the other hand, by Assumption~\ref{assumption:dn}, the LHS can be
lower bounded as
\begin{equation}
  \label{equation:dn-lower-bound}
  \sum_{i=1}^k \<Q_1^\top A_i Q_1, V_\Delta V_\star^\top\>^2
  = \sum_{i=1}^k \<Q_1^\top A_iQ_1, \frac{V_\Delta V_\star^\top +
    V_\star V_\Delta^\top}{2}\>^2 \ge
  \sigmadn \lfro{\frac{V_\Delta V_\star^\top +
    V_\star V_\Delta^\top}{2}}^2 \ge \frac{\sigmadn}{2}\lfro{V_\Delta
  V_\star^\top}^2,
\end{equation}
where we symmetrized the matrix $V_\Delta V_\star^\top$ and used the
fact $\lfro{(A+A^\top)/2}^2=(\lfro{A}^2+\tr(A^2))/2\ge \lfro{A}^2/2$
for any square matrix $A$ to get the last inequality.

The quantity $\lfro{V_\Delta V_\star^\top}^2$ is invariant under any
right orthogonal rotation $V_\star\mapsto V_\star\Omega$,
$V_\Delta\mapsto V_\Delta\Omega$. Note that $r\ge r_\star$, we can
then find an appropriate rotation under which
$V_\star=[V_{\star,1}, \zeros_{r_\star\times(r-r_\star)}]$ where
$V_{\star,1}\in\R^{r_\star\times r_\star}$. Let
$V_\Delta=[V_{\Delta,1}, V_{\Delta,2}]$ under
the same rotation. The requirement $V_\Delta^\top V_\star\succeq 0$
then reads
\begin{equation*}
  V_\Delta^\top V_\star = \begin{bmatrix}
    V_{\Delta,1}^\top V_{\star,1} & 0 \\
    V_{\Delta,2}^\top V_{\star,1} & 0
  \end{bmatrix} \succeq 0.
\end{equation*}
In particular, $V_{\Delta,2}^\top V_{\star,1} = 0$, and as $V_{\star,1}$ is
invertible, we get $V_{\Delta,2}=0$ and $V_\Delta=[V_{\Delta,1},
0]$. Thus $V_\Delta V_\star^\top=V_{\Delta,1}V_{\star,1}^\top$, and
\begin{align*}
  & \quad \lfro{V_\Delta V_\star^\top}^2 =
    \lfro{V_{\Delta,1}V_{\star,1}^\top}^2 =
    \<V_{\Delta,1}^\top V_{\Delta,1}, V_{\star,1}^\top V_{\star,1}\>
  \\
  & \ge \lambda_{\min}\left(V_{\star,1}^\top V_{\star, 1}\right) \cdot 
    \tr\left(V_{\Delta,1}^\top V_{\Delta,1}\right) =
    \lambda_{r_\star}\left(V_\star V_\star^\top\right) \cdot
    \lfro{V_{\Delta,1}}^2 =
    \lambda_{r_\star}\left(V_\star V_\star^\top\right) \cdot
    \lfro{V_\Delta}^2.
\end{align*}
Substituting into~\eqref{equation:dn-lower-bound} and combining with
the upper bound~\eqref{equation:dn-upper-bound}, we get
\begin{equation*}
  \frac{\sigmadn}{2} \lambda_{r_\star}\left(V_\star
    V_\star^\top\right) \cdot 
  \lfro{V_\Delta}^2 \le 2\left(\sum_{i=1}^k \lfro{Q_2^\top
      A_iQ_1}^2\right) \cdot 
  \lambda_{\max}\left(V_\star V_\star^\top\right) \cdot
  \lfro{W_\Delta}^2 + 2C(\lfro{V_\Delta}^2+\lfro{W_\Delta}^2)^2.
\end{equation*}
Noting that $X_\star=U_\star U_\star^\top=Q_1V_\star
V_\star^\top Q_1^\top$, we get that
\begin{equation*}
  \lfro{V_\Delta}^2 \le \kappa\lfro{W_\Delta}^2 + C'(\lfro{V_\Delta}^2
  + \lfro{W_\Delta}^2)^2,~~~
  \kappa \defeq \frac{4\lambda_1(X_\star)}{\lambda_{r_\star}(X_\star)}
  \cdot \frac{\sum_{i=1}^k \lfro{Q_2^\top A_iQ_1}^2}{\sigmadn}.
\end{equation*}
When $U\in\ball(U_\star,\eps)$, we have
$\lfro{\Delta}^2=\lfro{V_\Delta}^2+\lfro{W_\Delta}^2\le \eps^2$, which
gives 
\begin{equation*}
  \lfro{V_\Delta}^2 \le \kappa\lfro{W_\Delta}^2 +
  C'\eps^2(\lfro{V_\Delta}^2+\lfro{W_\Delta}^2)~~~{\rm
    or}~~~\lfro{V_\Delta}^2 \le
  \frac{\kappa+C'\eps^2}{1-C'\eps^2}\lfro{W_\Delta}^2.
\end{equation*}
Thus for sufficiently small $\eps$, we have $\lfro{V_\Delta}^2\le
2\kappa\lfro{W_\Delta}^2$. 
This will imply the desired quadratic growth:
\begin{equation*}
  \vphi(U) = \vphi(U_\star+\Delta) = \lfro{W_\Delta}^2 \ge
  \frac{1}{2\kappa+1} (\lfro{V_\Delta}^2+\lfro{W_\Delta}^2) =
  \frac{1}{2\kappa+1}\lfro{\Delta}^2.
\end{equation*}

\begin{remark}
  The constant $\kappa$ above is the product of two condition numbers:
  the condition number of the solution $X_\star$, and the ``condition
  number'' of the constraint $\mc{A}$, measured in a way depending on
  $X_\star$ (through $Q_1,Q_2$).
\end{remark}

\begin{remark}
  A quadratic function on a manifold near a strict minimum does not
  necessarily have quadratic growth. For example, on the curve
  $y=x^2$, the quadratic function $f(x,y)=y^2$ has a strict minimum
  $(0,0)$, but there does not exist $c>0$ such that
  \begin{equation*}
    y^2 \ge c(x^2+y^2) = c(y+y^2)
  \end{equation*}
  near the origin. This is why simply showing $\lfro{Q_2^\top U}^2$
  non-vanishing is not enough. In contrast, in Euclidean spaces, a
  quadratic function with a strict minimizer always have quadratic
  growth, by a scaling argument.
\end{remark}






\subsubsection{Proof of Lemma~\ref{lemma:quadratic-normal-component}}
\label{section:proof-quadratic-normal-component}
Such a result holds locally (around $U_\star$) for any smooth manifold
$\mc{M}$; here we present a proof that harnesses the special structure
of $\mc{M}_r$, which allows the result to hold globally.

As $U,U_\star\in\mc{M}_r$, we have
$\mc{A}(UU^\top)=\mc{A}(U_\star U_\star^\top)=b$, so
\begin{equation*}
  0 = \<A_iU, U\> - \<A_iU_\star, U_\star\> = 2\<A_iU_\star, \Delta\>
  + \<A_i\Delta, \Delta\>.
\end{equation*}
Squaring and summing over $i$, we get
\begin{equation*}
  \sum_{i=1}^k \<A_iU_\star, \Delta\>^2 =
  \frac{1}{4}\sum_{i=1}^k\<A_i\Delta, \Delta\>^2 \le
  \frac{1}{4}\sum_{i=1}^k \opnorm{A_i}^2\lfro{\Delta}^4.
\end{equation*}
So it suffices to let $C=\sum_{i=1}^k\opnorm{A_i}^2/4$. 

\subsection{Proof of Theorem~\ref{theorem:sdp-linear-convergence}}
\label{appendix:proof-corollary-rate}
We need to first check that the proximity property holds, i.e. for
sufficiently small $\eps>0$, there exists some $d(\eps)>0$ such that
initializing in the $d(\eps)$-neighborhood of $S$, the iterates never
leave the $\eps$-neighborhood.

Recall that the exact penalty objective
$h(X)=f(X)+\lambda\ltwo{\mc{A}(X)-b}$ is convex in $X$.
Fixing any $\eps>0$, by Lemma~\ref{lemma:convex-growth}, there exists
$\delta(\eps)>0$ such that for any $X\succeq 0$,
$h(X)-h_\star\le\delta(\eps)$ implies $\lfro{X-X_\star}\le\eps$. Now,
suppose we initialize at $R^0$ such that
$\vphi(R^0)-\vphi_\star\le\delta(\eps)$, then by the descent property,
we have $\vphi(R^k)-\vphi_\star\le\delta(\eps)$ for all $k$, or
letting $X^k=R^kR^{k\top}$,
\begin{equation*}
  h(X^k) - h_\star = \vphi(R^k) - \vphi_\star \le \delta(\eps),
\end{equation*}
which implies that $\lfro{X^k - X_\star}\le\eps$. Further, in either
the exactly-specified or over-specified case, by
Lemma~\ref{lemma:overspecified-growth}, we have $\dist(R^k,S) \le
O(\sqrt{\eps})$. Thus, the iterates stay in a $O(\sqrt{\eps})$
neighborhood as long as $\vphi(R^0) - \vphi_\star\le\delta(\eps)$,
which can be guaranteed if $\dist(R^0, S)\le \delta(\eps)/\lipphif$
by Lipschitzness. Hence, the proximity property is satisfied with
$d(\eps)=\delta(\eps)/\lipphif$.

Theorem~\ref{theorem:geometry-exact-specified}
or~\ref{theorem:linear-objective-over-specified} guarantees that
$\vphi$ satisfies the norm convexity property with $\ell=5$. One
could easily check that for our $\vphi(\cdot)=h(c(\cdot))$,
$h(X)=f(X)+\lambda\ltwo{\mc{A}(X)-b}$ is locally Lipschitz with
constant $L=L_f+\lambda\opnorm{\mc{A}}$, and $c(R)=RR^\top$ has
$\beta=2$ Lipschitz gradients. Plugging in these bounds into
Theorem~\ref{theorem:linear-convergence}, we get linear convergence
with the desired rate.

We now compute the $\Lambdaall$, the lowest allowed choice of
$\lambda$, in our factorized semidefinite problem. From the proof of
Theorem~\ref{theorem:quadratic-implies-regularity}, 
\begin{equation*}
  \Lambdaall = \max\{\Lambdaqg, \Lambdaproj\},
\end{equation*}
where $\Lambdaqg=2\lipphif/\sigmacq$. At $R_\star$, by the first-order
optimality condition, we have
\begin{equation*}
  \grad\vphi_f(R_\star) + M_{R_\star} y_\star = 0,
\end{equation*}
 so $\grad\vphi_f(R_\star) \le
 \sigma_{\max}(M_{R_\star})\ltwo{y_\star}$ and locally $\lipphif\le 2
 \sigma_{\max}(M_{R_\star})\ltwo{y_\star}$. For the other part, note
 that all $R_\star\in S$ have equivalent behaviors, so
\begin{equation*}
  \Lambdaproj = 2\sup_{R_\star\in S}\ltwo{z(R_\star)} =
  2\ltwo{z(R_\star)} = 2\ltwo{y_\star}.
\end{equation*}
Putting together, it suffices to take
\begin{equation*}
  \Lambdaall =
  \max\set{\frac{4\sigma_{\max}(M_{R_\star})\ltwo{y_\star}}{\sigmacq},
    2\ltwo{y_\star}} =
  \frac{4\sigma_{\max}(M_{R_\star})\ltwo{y_\star}}{\sigmacq}.
\end{equation*}
With $\lambda=\Lambdaall$, the rate is
\begin{equation*}    
  q = 1 - \frac{1}{9+40\ell+100\ell^2L\beta/\alpha_\vphi} = 1 -
  \frac{1}{209 + 5000(\lipf +
    4\opnorm{\mc{A}}\ltwo{y_\star}\sigmamax/\sigmacq)/\alpha_\vphi}.
\end{equation*}
This is the desired result. \qed

{\bf Remark} We might carefully compare the rate we get with what we
would expect in classical non-linear programming or Riemannian
optimization. To achieve $\eps$-accuracy, locally we need
$O(M\log\frac{1}{\eps})$ iterations, where $M$ is the rate constant
\begin{equation*}
  M = O\left( \frac{L_f + \opnorm{\mc{A}}
      \ltwo{y_\star}\sigmamax/\sigmacq}{\alpha_\vphi} \right).
\end{equation*}
Most quantities have fairly standard interpretations:
\begin{enumerate}[(1)]
\item $\sigmamax/\sigmacq$: the condition number of the gradient
  $\grad(\mc{A}(RR^\top)-b)|_{R=R_\star}$. This
  term represents the difficulty to project onto the tangent space of
  the feasible set.
\item $\alpha_\vphi$: the quadratic growth constant of $\vphi_f$ on
  the feasible set, or the minimum eigenvalue of
  $\rhess\vphi_f(R_\star)$. This is the analogue of the strong
  convexity parameter in Euclidean optimization.
\item $\opnorm{\mc{A}}$: this is a scaling factor of the constraint
  function that can be compared to $f$.
\item $\ltwo{y_\star}$: norm of the optimal dual variable, importance
  of the constraint.
\end{enumerate}
Now, what does the term $L_f$ stand for? It is the Lipschitz constant
of the objective $f$, but turns out that it is also \emph{part of the
  Hessian of the factorized problem}. In fact, one can check that
\begin{equation*}
  \grad^2\vphi_f(R)[\Delta,\Delta] = \underbrace{
  \grad^2f(RR^\top)[R\Delta^\top+\Delta R^\top,R\Delta^\top+\Delta
  R^\top]}_{\rm I} + \underbrace{\<\grad f(RR^\top),
  \Delta\Delta^\top\>}_{\rm II}. 
\end{equation*}
The term $L_f$ bounds $\opnorm{\grad f(RR^\top)}$,  part II of the
Hessian. As the algorithm solves
problem~\eqref{problem:constrained-bm}, the complexity must also
depend on the maximum eigenvalue of part I -- this is taken care of in
the subproblems.

\section{Experiments}
\label{section:experiments}
We perform experiments that corroborate our theoretical
observations. We begin by describing our implementation of the
prox-linear algorithm and then show experimental results on a set of
MaxCut SDP problem and a set of quadratic semidefinite optimization
problem. 

\subsection{Implementation details}
The prox-linear algorithm minimizes $\vphi(R)$ by iteratively solving
the prox-linear sub-problems~\eqref{algorithm:prox-linear}, which
we recall minimizes (in $\Delta$)
\begin{align*}
  & f(c(R)+\grad c(R)^\top\Delta) + \lambda\ltwo{\mc{A}(c(R)+\grad
    c(R)^\top\Delta)-b} + \frac{1}{2\alpha}\lfro{\Delta}^2 \\
  = & f(RR^\top + R\Delta^\top + \Delta R^\top) +
      \lambda\ltwo{\mc{A}(RR^\top) + 2\mc{A}(R\Delta^\top)-b} +
      \frac{1}{2\alpha}\lfro{\Delta}^2. 
\end{align*}
These problems are $1/\alpha$-strongly convex but non-smooth due to
the norm term. Utilizing that $f$ and the norm takes in a linear
transform of $\Delta$, we describe how we use a variant of ADMM, the
Proximal Operator Graph Splitting (POGS) algorithm~\cite{ParikhBo14},
to solve these problems.

Introducing $z=2\mc{A}(R\Delta^\top)\in\R^m$, the problem is
equivalent to
\begin{align*}
  \minimize & ~~ f(RR^\top+R\Delta^\top+\Delta R^\top) +
              \lambda\ltwo{z+\mc{A}(RR^\top)-b} +
              \frac{1}{2\alpha}\lfro{\Delta}^2 \\
  \subjectto & ~~ z = 2\mc{A}(R\Delta^\top).
\end{align*}
Let $g(z)=\lambda\ltwo{z+\mc{A}(RR^\top)-b}$, and
$h(\Delta)=f(RR^\top+R\Delta^\top+\Delta
R^\top)+\frac{1}{2\alpha}\lfro{\Delta}^2$. Let
$\mc{C}=\set{(z,\Delta):z=2\mc{A}(R\Delta^\top)}$
be the constraint set. The POGS algorithm contains the following steps
in each iterate.
\begin{enumerate}[(1)]
\item Compute
  $(z^{k+1/2},\Delta^{k+1/2})=(\prox_{g/\rho}(z^k-\lambda_z^k),
  \prox_{h/\rho}(\Delta^k-\lambda_\Delta^k))$.
\item Compute
  $(z^{k+1},\Delta^{k+1})=\proj_{\mc{C}}(z^{k+1/2}+\lambda_z^k,
  \Delta^{k+1/2}+\lambda_\Delta^k)$.
\item Update $\lambda_*^{k+1}=\lambda_*^k+(*^{k+1/2}-*^{k+1})$ for
  $*\in\set{z,\Delta}$.
\end{enumerate}
In step (1), we need to evaluate the prox operator of $g$ and
$h$. As $g$ is a two-norm, the prox mapping has an
explicit solution. Evaluating the prox of $h$ is in general
non-trivial, but in case where $f$ is linear or quadratic, it is
straightforward to find the zero of the gradient, which by
convexity is the solution to the proximal problem. In step (2), we
need to compute the projection onto a linear subspace in the space of
$(z,\Delta)$. For small instances, we can directly form the matrix of
the linear mapping and compute the projection matrix. For large
instances, the inverse involved becomes too costly, but it is often
efficient to compute $z=2\mc{A}(R\Delta^\top)$ given $\Delta$, in
which case we can use conjugate gradient algorithms to compute the
projection iteratively. 

\subsection{Linear objective}
\label{section:experiments-linear}
In our first set of experiments, we solve a MaxCut SDP with $n=50$
and $\rank(X_\star)=r_\star=2$. Such a problem instance is obtained by
explicitly constructing a set of dual certificate conditions.

\paragraph{Methods}
We test our prox-liner algorithm
(Iteration~\eqref{algorithm:prox-linear}) on the exact penalty
formulation~\eqref{problem:penalized-bm} with factorization rank
$r\in\set{1,2,3}$. The sub-problems are solved by the POGS
algorithm. We set the penalty parameter $\lambda=100>2\ltwo{y_\star}$
sufficiently large. As suggested by our theory, we should choose the
stepsize $t=(L\beta)^{-1}\propto\lambda^{-1}$. We set $t=1/\lambda$,
which we observe often yields good convergence in practice. We
initialize randomly at points that are not necessarily close to $S$ --
this also helps us examine the global behavior of our algorithm.

To compare with Riemannian algorithms, we run another set with
feasible start, and compare our prox-linear with the projected
gradient descent algorithm, which in the MaxCut SDP iterates
\begin{equation*}
  R^{k+1} = {\rm row-normalize}\left((I_{n\times n} - 2tC)R^k\right).
\end{equation*}
This has similar behavior as the Riemannian gradient descent but is
slightly easier to implement. 
\begin{figure}[ht]
  \centering
  \includegraphics[width=\linewidth]{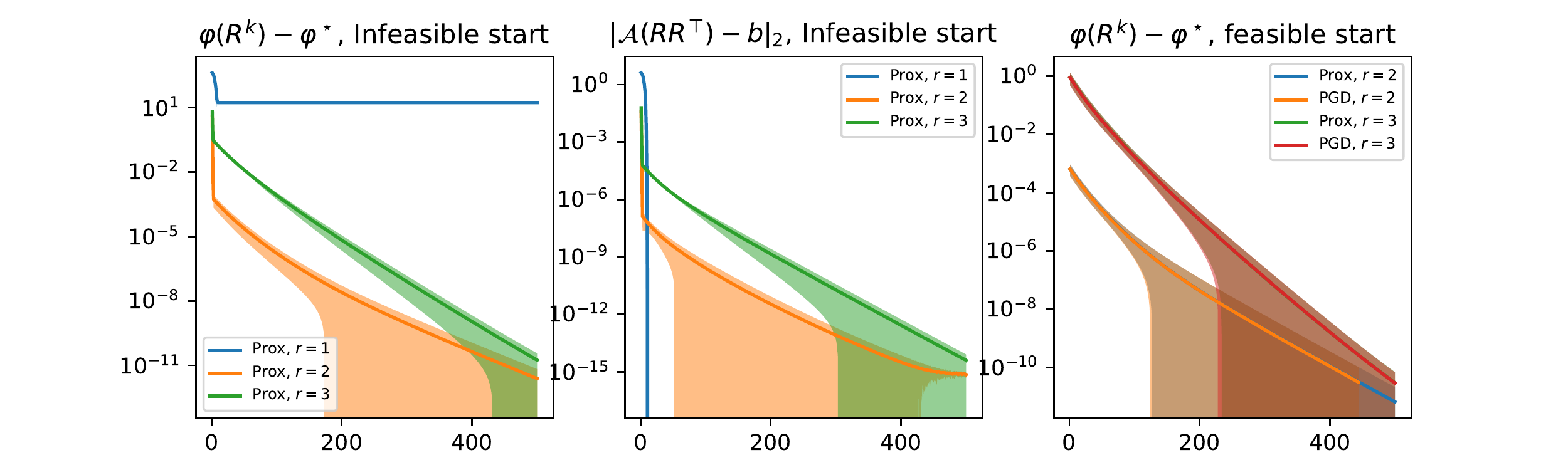}
  \caption{Results on the MaxCut SDP. Left and Middle: infeasible
    start. Right: feasible start. The rank of the true solution is
    2. The algorithm converges linearly to $X_\star$ when factorizing
    to rank $2,3$ and fails to find $X_\star$ when factorizing
    to rank $1$.} 
  \label{figure:maxcut}
\end{figure}

\paragraph{Result}
Figure~\ref{figure:maxcut} reports the composite objective value and
the infeasibility, plotting their mean and error bar of $\pm 2$ standard
deviations across $n_{\rm init}=10$ random initializations.
Observe that $r\in\set{2,3}$ yields linear convergence to the optimum
with $r=2$ (true rank) converging faster and $r=3$ (overspecified
rank) slightly slower. As the objective has linear growth, our theory
guarantees that the factorized versions have quadratic growth, and the
experiments are consistent with the theory. In contrast, when we
under-specifying the rank (setting $r=1$), the exact penalty
formulation will converge to a sub-optimal point with large objective
value. Also, the prox-linear algorithm converges very fast in the
beginning iterations, enforcing $R$ to quickly fall onto the
constraint set. This suggests that it also has nice global behavior.

In the feasible start experiment, we see that projected gradient
descent gives almost the same objective values as the exact penalty
approach, suggesting that the connections we make in
Theorem~\ref{theorem:quadratic-implies-regularity} might indeed
govern the behavior of these two algorithms. Formally making this
connection remains an interesting open question. 

\subsection{Quadratic objective}
In our second set of experiments, we solve the following random
quadratics problem.

\paragraph{Problem setting}
For $Y\in\symm^n$, consider the approximation problem
\begin{equation}
  \label{problem:random-quadratics}
  \begin{aligned}
    \minimize & ~~ \frac{1}{2}\lfro{X-Y}^2 \\
    \subjectto & ~~ \mc{A}(X) = b \\
    & ~~ X \succeq 0.
  \end{aligned}
\end{equation}
When does this problem have a low-rank solution? Intuitively, we
would ask the constraint $\mc{A}(X)=b$ to be restrictive -- without
the constraints, the solution will be the projection of $Y$ onto the
PSD cone, whose rank equals the number of positive eigenvalues of
$Y$.

We now spell out conditions under which a given
$X_\star=R_\star R_\star^\top$ is the unique solution and the
problem has quadratic growth. By the KKT conditions,
$X_\star\succeq 0$ is optimal iff $\mc{A}(X_\star)=b$ and there
exists dual pairs $(y,Z)$ such that
\begin{equation*}
  X_\star - Y + \mc{A}^*(y) = Z,
\end{equation*}
with $Z\succeq 0$ and $\<Z,X_\star\>=0$. Letting
$Q_1\in\R^{n\times r_\star}$ be a basis of ${\rm range}(R_\star)$
and $Q_2\in\R^{n\times(n-r_\star)}$ be its orthogonal complement,
then we have $Z=Q_2WQ_2^\top$ for some $W\succeq 0$. One can further
check that if $W\succeq\lambda I_{n\times r_\star}$ for some
$\lambda>0$, then the problem has quadratic growth.

We note that it is fairly straighforward to generate an instance of
problem~\eqref{problem:random-quadratics} with random
constraints that has low-rank solution and quadratic growth. The
following gives such a procedure.
\begin{enumerate}[(1)]
\item Generate $R_\star\in\R^{n\times r}$ and $X_\star=R_\star
  R_\star^\top$. Find $Q_1,Q_2$ and set
  $Z=\lambda Q_2Q_2^\top$ for some $\lambda>0$.
\item Generate $y\in\R^k$, generate $\mc{A}$ randomly,
  e.g. $A_i=a_ia_i^\top$ where $a_i$ are i.i.d. standard
  Gaussian. Compute $b=\mc{A}(X_\star)$.
\item Compute $Y=X_\star+\mc{A}^*(y)-Z$.
\end{enumerate}

\paragraph{Method}
We solve an instance of this random quadratics problem with $n=50$ and
$\rank(X_\star)=r_\star=3$. We specify
the rank $r\in\set{3,4}$ and run the prox-linear algorithm with the
general-case POGS to solve the sub-problems. We choose $\lambda=500$
sufficiently large and $t=1/\lambda$ accordingly such that prox-linear
converges linearly with $r=3$.

\paragraph{Result}
Figure~\ref{figure:random-quadratics} reports the compositie objective
values and infeasiblities, plotting their mean and error bar across
$n_{\rm init}=10$ random (infeasible) initializations.  The
prox-linear algorithm converges linearly to a very high accuracy
(about $10^{-12}$ feasibility). Observe that when we over-specify the
rank (setting $r=4$), the algorithm no longer converges linearly. Note
that the $r=4$ experiments are with the same $(\lambda,t)$ values as
the $r=3$ experiment, but we find that this phenomenon is very robust
to the penalty parameter and stepsize choice. This demonstrates our
theory about rank over-specification in
Section~\ref{section:over-specification}: over-specifying the rank for
general non-linear objectives will not carry quadratic growth to
$\vphi$ and thus prohibits linear convergence of the prox-linear
algorithm.

\begin{figure}[ht]
  \centering
  \includegraphics[width=0.7\linewidth]{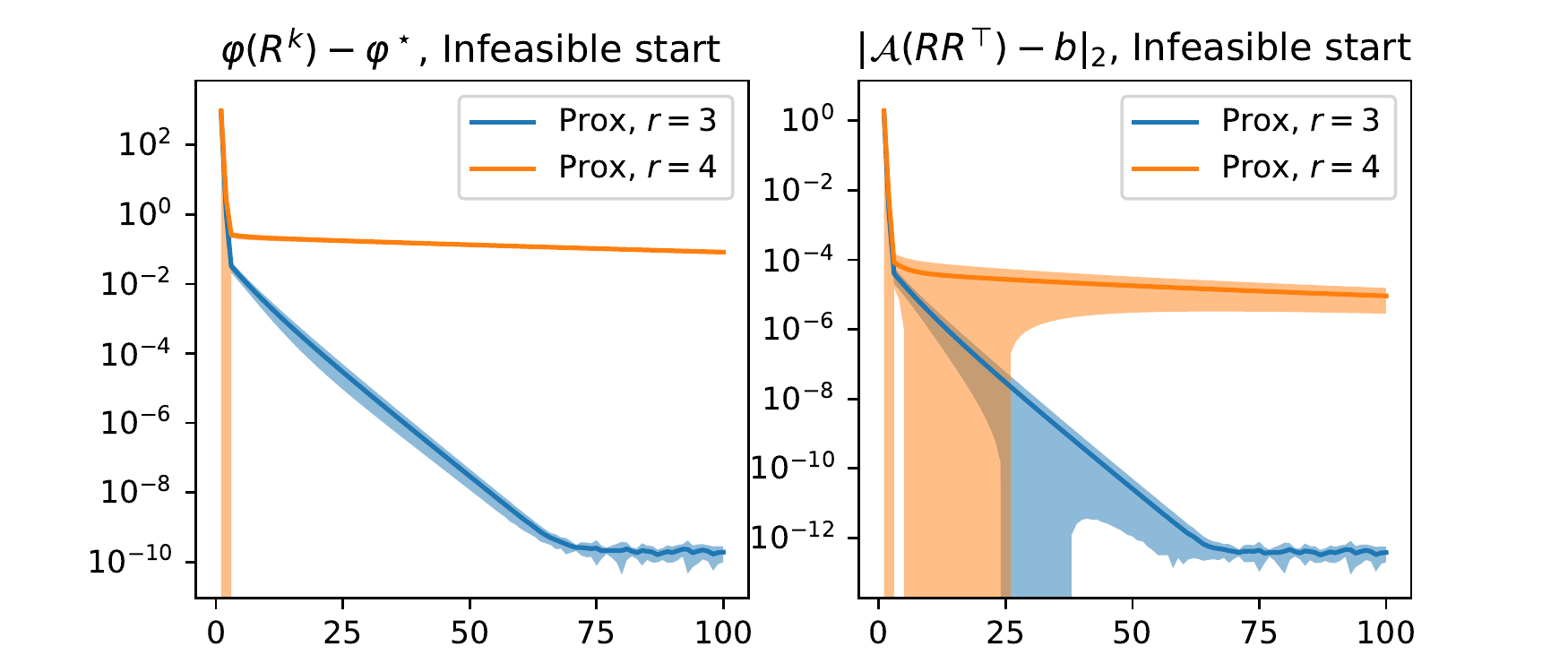}
  \caption{Results on the random quadratics problem. Left: composite
    objective. Right: feasibility gap. The solution has rank 3. The
    prox-linear algorithm converges linearly to $X_\star$ when
    factorizing to the right rank, and much slower when over-specifing
    the rank by one. 
  }
  \label{figure:random-quadratics}
\end{figure}

\end{document}